\documentclass[11pt]{article}
\usepackage[a4paper,margin=1in]{geometry}
\usepackage[T1]{fontenc}
\usepackage[utf8]{inputenc}
\usepackage{lmodern}
\usepackage{amsmath,amssymb,amsthm,mathtools,bm}
\usepackage{booktabs,array,microtype}
\usepackage{enumitem}
\usepackage[hidelinks]{hyperref}
\usepackage[nameinlink,capitalize]{cleveref}

\newtheorem{theorem}{Theorem}[section]
\newtheorem{proposition}[theorem]{Proposition}
\newtheorem{lemma}[theorem]{Lemma}
\newtheorem{corollary}[theorem]{Corollary}
\theoremstyle{definition}
\theoremstyle{definition}
\newtheorem{definition}[theorem]{Definition}
\theoremstyle{remark}
\newtheorem{remark}[theorem]{Remark}
\newtheorem{warning}[theorem]{Warning}

\newcommand{\R}{\mathbb R}
\newcommand{\dd}{\,d}

\title{\textbf{Endpoint Structure in Three-Dimensional Navier--Stokes}}
\author{Rishad Shahmurov\\\small Cellular Products Research and Development, Roswell, GA, USA}
\date{September 2026}

\begin{document}
\maketitle

\begin{abstract}
We study possible singular endpoints of the three-dimensional incompressible Navier--Stokes equations by parabolic rescaling.  Weak limits can lose information through pressure, nonlinear products, concentration, spatial tails, or symmetry, so we keep these quantities as part of one augmented endpoint state and use a fixed whole-space pressure representative throughout the shrinking chain.  We prove compact endpoint realization, nonoverlapping assignment of positive defects, two-generation inheritance, rotational alternatives, and a high-critical kinetic normalization centered at the original singular point.  We show that positive rotational distance alone does not coerce a productive rotational Reynolds stress, and we give an explicit conditional compactness theorem under which frequency-tight amplitude-fast records pass strongly to a nonzero ancient Euler profile.  We then study a possible Euler limit that is invisible under positive heat evolution.  Exact heat-orbit nullity forces cancellation separately on each interaction-energy shell and isotropic covariance on every Fourier sphere; independently, any nonzero $L^2$ field normalized along its long heat orbit loses all mass from every fixed spatial ball.  Frequency anti-concentration, the accumulating heat-chain passage, and retention of a singular witness remain open; no unconditional regularity theorem is claimed.
\end{abstract}

\section{Introduction}

The regularity problem for the three-dimensional incompressible
Navier--Stokes equations is often approached by assuming that a first
singularity occurs and then magnifying the solution near that point.  The
parabolic scaling of the equations makes this natural: a shrinking physical
region can be rescaled to a fixed cylinder, and one may then try to extract a
nontrivial limiting flow.  This viewpoint goes back to the classical weak
solution theory of Leray \cite{Leray1934} and is central to the partial
regularity theory of Caffarelli, Kohn, and Nirenberg \cite{CKN1982}.  In many
modern arguments the limiting object is ancient or terminal, and one seeks to
show that its structure is incompatible with the existence of the original
singularity; Type-I compactness and Liouville methods provide an important
example of this strategy \cite{AlbrittonBarker2019}.

The difficulty is that taking a weak limit does not automatically preserve all
of the information that was important before rescaling.  The velocity may
converge while part of a nonlinear product is lost.  Energy may escape through
a spatial collar or a time face.  Pressure is nonlocal, so a local pressure
decomposition can change when the cylinder changes.  Directional information
may oscillate even when the magnitude remains controlled.  A sequence can also
concentrate in amplitude or drift toward a different spatial scale.  If these
possibilities are not recorded explicitly, a compactness argument can appear
to close simply because the missing information has disappeared from the
notation.

The purpose of this paper is to make that loss of information visible.  We do
not try to prove that every possible Navier--Stokes singularity has one
particular geometric form.  Instead, we construct a selected endpoint
framework in which the quantities needed by later arguments have a definite
place to go.  The basic principle is simple: the limiting velocity and pressure
are retained together with the defects, tails, concentration measures, and
finite-dimensional quotient data that may survive the limiting process.  The
result is an endpoint state that is larger than the usual pair \((u,p)\), but
small enough that its components can still be controlled by standard
compactness and measure arguments.

A first point of emphasis is pressure.  In whole space the pressure is
recovered from the velocity by the Riesz transforms, up to an irrelevant
function of time.  We therefore fix one global representative
\[
  p^\sharp=\mathcal R_i\mathcal R_j(u_i u_j)
\]
throughout the entire shrinking chain.  Local near/far pressure decompositions
are always taken from this same representative.  This prevents a different
additive normalization from being introduced at every scale.  Such a choice is
especially important when pressure contributions are to be compared or
assigned along a nested sequence.  Our use of a fixed representative is
compatible with the local pressure expansion viewpoint developed, for example,
by Bradshaw and Tsai \cite{BradshawTsai2022}.

A second point is the distinction between positive quantities and signed
currents.  Local energy defects and concentration measures can be treated as
nonnegative Radon measures.  They may therefore be assigned to the first scale
at which they appear, which prevents the same positive quantity from being
charged repeatedly along a nested chain.  Signed fluxes and pressure-triad
currents behave differently: cancellation is part of their meaning, so they
are retained as signed distributions or currents rather than inserted into a
positive budget.  This separation is elementary, but it is essential for the
quantitative statements later in the paper.

The first group of results establishes the compact endpoint itself.  Under the
usual scale-invariant local energy, dissipation, and pressure bounds, together
with tightness of the explicitly retained defect coordinates, a normalized
sequence has a subsequence converging to a terminal endpoint state.  The
velocity converges strongly below the natural exponent, the near pressure
converges locally after the fixed normalization, and all weak-product,
concentration, tail, and quotient losses are recorded separately.  We then
show that genuinely nonnegative quantities admit a first-assignment rule and a
corresponding quantization estimate.  These statements are designed to answer
a basic bookkeeping question: when a positive amount of critical information
survives at many scales, which part is actually new and which part has already
been counted?

The second group of results concerns inheritance and detection.  We retain two
successive generations of the selected endpoint state and use a fixed family of
five trace-free symmetric receivers to detect the full local current.  The
five-frame device is intentionally separated from the two-column Gram matrix:
the former detects the current itself, whereas the latter is only a conditioning
tool for comparing two inherited observations.  This distinction avoids
confusing failure of a particular two-dimensional test with disappearance of
the underlying current.

The final part of the paper studies what the endpoint can look like when the
critical quantity is either bounded or unbounded.  On bounded critical
sequences, compactness leads to a rotational alternative: one may approach an
axisymmetric class, retain a genuinely non-axisymmetric compact state, or reach
a nonsmooth terminal state with a positive gap from the relevant rotational
Killing fields.  The connection with rotational geometry is consistent with
the broader role of vorticity direction in regularity theory, going back in
particular to Constantin and Fefferman \cite{ConstantinFefferman1993}.  On the
unbounded branch, we prove a different alternative: either exterior stress or
tail information becomes non-negligible, or one can normalize at the original
singular point to obtain an amplitude-fast kinetic record.  The spatial mark is
therefore retained rather than being lost by recentering.

We also isolate what would happen after a further strong passage to an Euler
profile whose positive heat orbit remains stationary.  In Fourier variables the
heat weight separates the quadratic interaction by the value of
\(|\eta|^2+|\xi-\eta|^2\).  Vanishing for an interval of heat times therefore
forces cancellation on each such interaction shell separately; a small-output
limit then gives an isotropy law on every Fourier sphere.  A complementary
spectral theorem shows that, for any nonzero whole-space \(L^2\) field, the
long-time heat orbit normalized in \(L^2\) leaves every fixed physical ball.
Thus a heat-orbit-null endpoint cannot remain both nontrivial and spatially
tight under arbitrarily large heat ages.  This conclusion is conditional when
applied to the Navier--Stokes record: the strong nonlinear compactness and
accumulating heat-chain passage needed to reach that Euler class are not proved
here.

These results deliberately stop before a regularity theorem.  The bounded
branch still contains a nonsmooth positive-Killing-gap terminal state that is
not excluded here; an explicit local rotational-stress example below shows why
positive distance from the axisymmetric subspace cannot by itself supply a
productive mean forcing.  On the high-critical branch, the amplitude-fast
record does not by itself prevent spatial or frequency concentration.  We prove
that localized frequency tightness together with a weak time-compactness bound
is sufficient for strong local passage to an ancient Euler profile, but the
paper does not establish those hypotheses from the singularity record and does
not prove that a singular or deformation witness survives in the limit.  Thus
the paper should be read as an endpoint-structure and compactness result: it
narrows and organizes the possible terminal behavior, but it does not claim a
universal reduction to axisymmetric flow or an unconditional solution of the
three-dimensional regularity problem.

The paper is organized as follows.  We first fix the pressure normalization and
the selected scales, then define the augmented endpoint package and prove its
compact realization.  We next develop first assignment, quantization, and the
chainwise positive ledger, followed by two-generation inheritance and fixed-frame
current detection.  The last sections treat rotational compactness and high-critical
records, then derive the heat-orbit shellwise rigidity and normalized
spectral-recession theorems, before collecting the remaining open problems and
pressure/topology conventions.

\section{Pressure normalization and selected scales}

Consider a smooth finite-energy solution before a hypothetical first singular
time \(T_*\), or a suitable weak limit obtained from such a sequence.  In
whole space we fix the physical pressure representative once and for all by
\[
  p^\sharp(t)=\mathcal R_i\mathcal R_j(u_i u_j)(t)
\]
in the distributional sense, with the purely time-dependent gauge fixed
globally.  Every local pressure decomposition is taken from this same
representative.

\medskip
\noindent\textbf{Root-gauge parent--descendant pressure identity.}
\begin{proposition}
Let \(\eta_Q\) be a nested family of local cutoffs and put
\[
  p_Q^{\rm near}=\mathcal R_i\mathcal R_j(\eta_Q u_i u_j),
  \qquad
  h_Q=p^\sharp-p_Q^{\rm near}.
\]
On a descendant \(Q'\Subset Q\) on which both cutoffs are valid,
\[
  h_{Q'}-h_Q
  =
  \mathcal R_i\mathcal R_j\bigl((\eta_Q-\eta_{Q'})u_i u_j\bigr).
\]
After subtracting a fixed finite-dimensional harmonic jet, the only additional
term is the corresponding finite-dimensional quotient difference.
\end{proposition}

\begin{proof}
We give the argument at the level of distributions, because that is the
natural setting in which the whole-space pressure is fixed.  Denote by
\(
\mathcal P(f):=\mathcal R_i\mathcal R_j(f_{ij})
\)
the double Riesz transform applied to a matrix field.  By definition,
\[
 p_Q^{\rm near}=\mathcal P(\eta_Q u\otimes u),
 \qquad
 h_Q=p^\sharp-\mathcal P(\eta_Q u\otimes u),
\]
and the same formulas hold with \(Q\) replaced by \(Q'\).  Since the
whole-space representative \(p^\sharp\) is the same in both decompositions,
we may subtract the two identities without introducing any time-dependent
pressure gauge.  This gives
\[
\begin{aligned}
 h_{Q'}-h_Q
 &=\Bigl(p^\sharp-\mathcal P(\eta_{Q'}u\otimes u)\Bigr)
   -\Bigl(p^\sharp-\mathcal P(\eta_Qu\otimes u)\Bigr)\\
 &=\mathcal P\bigl((\eta_Q-\eta_{Q'})u\otimes u\bigr).
\end{aligned}
\]
This is exactly the asserted identity.  Notice that the right-hand side is
supported, before application of the nonlocal Riesz operator, where the two
cutoffs differ.  Thus the parent--descendant pressure mismatch is determined
by the physical velocity together with the already chosen cutoffs; there is
no additional local pressure datum.

We now include the finite-dimensional quotient used later in the endpoint
package.  Let \(\Pi_Q h_Q\) and \(\Pi_{Q'}h_{Q'}\) denote the prescribed
harmonic jets, or more generally the prescribed finite-dimensional quotient
coordinates, removed from the two harmonic remainders.  Then
\[
\begin{aligned}
 &(h_{Q'}-\Pi_{Q'}h_{Q'})-(h_Q-\Pi_Qh_Q)\\
 &\qquad=
 \mathcal P\bigl((\eta_Q-\eta_{Q'})u\otimes u\bigr)
 +\Pi_Qh_Q-\Pi_{Q'}h_{Q'}.
\end{aligned}
\]
The last two terms belong to the fixed finite-dimensional quotient space.
Consequently quotient subtraction changes only the recorded quotient
coordinate.  It does not generate a new pressure source or alter the root
gauge.  This proves both assertions.
\end{proof}

The fixed root gauge is necessary for the following reason.
\begin{remark}
A density of the form
\[
  |p-c_Q(t)|^{3/2}\,\dd x\,\dd t
\]
does not define one Radon measure on the root cylinder if the normalizing
function \(c_Q(t)\) depends on the selected cylinder \(Q\).  Scale-dependent
constants are legitimate local quotient coordinates, but they cannot be used
inside a single additive first-allocation budget.  All Radon budgeting below uses
the fixed representative \(p^\sharp\).
\end{remark}

\section{The augmented endpoint package}

\medskip
\noindent\textbf{Selected endpoint package.}
\begin{definition}
On a normalized parabolic chamber the endpoint package is
\[
\boxed{
\mathcal E_{\rm NS}
=
\bigl(
 U,P^\sharp,\mathcal R,\mathcal E_{\rm loc},
 \mathcal D_{\rm flux},\mathcal P_{\rm quot},
 \mathcal T_{\rm tail/time},\mathcal O_{\rm own},
 \mathcal H_{\rm hol},\mathcal J_{\rm triad},
 \nu_\xi,\mu_{\rm conc},\mu_{3/2}
\bigr).
}
\]
Here:
\begin{enumerate}[label=(\roman*),leftmargin=*]
\item \((U,P^\sharp)\) is a suitable ancient or terminal weak limit in the fixed
      pressure gauge;
\item \(\mathcal R\) is the Reynolds/subgrid stress generated by weak products;
\item \(\mathcal E_{\rm loc}\ge0\) is the local-energy defect measure;
\item \(\mathcal D_{\rm flux}\) is the retained scale-localized energy-flux
      distribution (including any Duchon--Robert type commutator);
\item \(\mathcal P_{\rm quot}\) records finite-dimensional harmonic/affine
      pressure quotient data;
\item \(\mathcal T_{\rm tail/time}\) records spatial-tail and time-face loss;
\item \(\mathcal O_{\rm own}\) records first assignment of nonnegative currencies;
\item \(\mathcal H_{\rm hol}\) records the fixed-frame transition/holonomy
      discrepancy when it is not already a quotient or cutoff term;
\item \(\mathcal J_{\rm triad}\) is the pressure-projected local-triad current;
\item \(\nu_\xi\) is the vorticity-direction Young measure;
\item \(\mu_{\rm conc}\) is the amplitude/concentration measure;
\item \(\mu_{3/2}\) is the critical threshold-dissipation measure supplied by
      the companion paper.
\end{enumerate}
\end{definition}

\begin{warning}[Nonnegative budgets versus signed currents]
Only genuinely nonnegative Radon measures are placed in a finite first-assigned
budget.  The local-energy defect and Reynolds trace are nonnegative in their
standard defect formulations.  A Duchon--Robert commutator or a
pressure-projected triad current is retained as a distribution or
measure-valued current unless a separate theorem supplies positivity.  The
notation must not identify a signed flux current with the nonnegative
local-energy defect.
\end{warning}

\section{Faithful endpoint realization}

\medskip
\noindent\textbf{Compact endpoint realization with retained defect data.}
\begin{theorem}
Let \((u_n,p_n^\sharp)\) be a normalized sequence of suitable weak
Navier--Stokes solutions on a fixed parabolic chamber.  Assume uniform
scale-invariant local energy and dissipation bounds, local
\(L^{3/2}\)-pressure bounds in the fixed root gauge, local boundedness of the
critical threshold measures, and tightness/boundedness of the explicitly
retained nonnegative defect coordinates.  Then after a subsequence the
sequence generates an endpoint package \(\mathcal E_{\rm NS}\).  Moreover:
\begin{enumerate}[label=(\alph*),leftmargin=*]
\item \(u_n\to U\) strongly in \(L^q_{\rm loc}\) for every
      \(q<10/3\), after the standard suitable-solution compactness step;
\item the localized near pressure converges strongly in \(L^r_{\rm loc}\) for every
      \(1<r<5/3\), while harmonic/tail parts are retained separately;
\item nonnegative defect measures are weak-* compact and lower semicontinuous;
\item weak nonlinear-product loss is represented by \(\mathcal R\);
\item direction oscillation and amplitude concentration are represented by
      \(\nu_\xi\) and \(\mu_{\rm conc}\);
\item the fixed pressure projection may be applied before passage to the limit,
      so \(\mathcal J_{\rm triad}\) is retained distributionally.
\end{enumerate}
\end{theorem}

\begin{proof}
We divide the extraction into the velocity, pressure, measure, and signed-current
coordinates.  All compactness statements below are first proved on a strict
subchamber \(Q'\Subset Q\) and are then diagonalized over a countable exhaustion
of the normalized chamber.

\emph{Step 1: energy bounds and the subcritical spacetime exponent.}
The assumed local energy and dissipation bounds give, uniformly in \(n\),
\[
 u_n\quad\hbox{bounded in}\quad
 L^\infty(I';L^2(B'))\cap L^2(I';H^1(B'))
\]
for every \(Q'=B'\times I'\Subset Q\).  The three-dimensional Sobolev
embedding \(H^1(B')\hookrightarrow L^6(B')\), combined with interpolation in
space and time, yields the standard estimate
\[
 \|u_n\|_{L^{10/3}(Q')}
 \le C_{Q'}
 \|u_n\|_{L^\infty_tL^2_x(Q')}^{2/5}
 \|u_n\|_{L^2_tH^1_x(Q')}^{3/5}.
\]
Thus \((u_n)\) is uniformly bounded at the natural suitable-weak exponent
\(10/3\).

\emph{Step 2: a time-derivative bound.}
Choose a cutoff \(\chi\in C_c^\infty(Q)\) that is one on \(Q'\).  On the support
of \(\chi\), the equation can be written schematically as
\[
 \partial_t u_n
 =\nu\Delta u_n-\mathbb P\nabla\!\cdot(u_n\otimes u_n),
\]
where the Leray projection is harmless in negative Sobolev spaces.  The first
term is bounded in \(L^2(I';H^{-1}(B'))\).  Since
\(u_n\in L^{10/3}(Q')\), one has
\(u_n\otimes u_n\in L^{5/3}(Q')\), and therefore the nonlinear term is bounded
in a space of the form
\(L^{5/3}(I';W^{-1,5/3}(B'))\).  After lowering exponents if necessary, both
terms are bounded in one reflexive negative Sobolev space
\[
 \partial_tu_n\quad\hbox{bounded in}\quad
 L^\sigma(I';W^{-1,m}(B'))
\]
for some fixed \(\sigma>1\) and \(m>1\).  This is the compactness input needed
for the Aubin--Lions lemma.

\emph{Step 3: strong local velocity convergence.}
Aubin--Lions applied to
\(H^1(B')\Subset L^2(B')\hookrightarrow W^{-1,m}(B')\) gives, after passage to
a subsequence,
\[
 u_n\to U\qquad\hbox{strongly in }L^2(Q').
\]
The uniform \(L^{10/3}\) bound and interpolation then imply
\[
 u_n\to U\qquad\hbox{strongly in }L^q(Q')
 \quad\text{for every }1\le q<10/3.
\]
In particular, if \(2<q<10/3\), then
\[
 u_n\otimes u_n\to U\otimes U
 \qquad\hbox{strongly in }L^{q/2}(Q').
\]
This observation is useful because it identifies precisely which quadratic
velocity products are already compact and which additional nonlinear
coordinates must be retained only for weaker limiting procedures.

\emph{Step 4: the near pressure.}
Fix a spatial cutoff \(\eta\) supported strictly inside the root chamber and
identically one on the spatial projection of \(Q'\).  The near pressure is
\[
 p_n^{\rm near}=\mathcal R_i\mathcal R_j(\eta u_{n,i}u_{n,j}).
\]
Choose any \(1<r<5/3\) and set \(q=2r\); then \(2<q<10/3\).  By Step 3 the
localized product converges strongly in \(L^r\).  The Riesz transforms are
bounded on \(L^r(\R^3)\) for \(1<r<\infty\), hence
\[
 p_n^{\rm near}\to
 \mathcal R_i\mathcal R_j(\eta U_iU_j)
 \qquad\hbox{strongly in }L^r(Q').
\]
The portion of \(p_n^\sharp\) not contained in this near field is harmonic on
the smaller spatial region, up to the explicitly retained tail and time-face
contributions.  Those pieces are therefore kept as separate endpoint
coordinates rather than being absorbed into the near-pressure convergence.
The use of the fixed root representative is essential here: no scale-dependent
additive pressure constant is introduced during extraction.

\emph{Step 5: nonnegative defect measures.}
Let \((\mu_n)\) denote any of the nonnegative Radon coordinates for which the
hypotheses give a uniform mass bound on compact subsets.  By the Riesz
representation theorem and Banach--Alaoglu, the unit ball of finite Radon
measures is weak-* compact on every compact subchamber.  Hence, after a
subsequence,
\[
 \mu_n\stackrel{*}{\rightharpoonup}\mu.
\]
For every nonnegative \(\varphi\in C_c\),
\[
 \int\varphi\,d\mu
 \le \liminf_{n\to\infty}\int\varphi\,d\mu_n
\]
whenever the coordinate arises from a lower-semicontinuous convex energy or
dissipation density.  This gives the required lower semicontinuity for the
local-energy, concentration, and critical-threshold coordinates.

\emph{Step 6: Reynolds/product defects and generalized oscillation data.}
Whenever a nonlinear quantity is already covered by the strong convergence in
Step 3, its defect is zero on \(Q'\).  For any additional bounded nonlinear
observable for which only weak convergence is available, we define its defect
as the difference between the weak limit of the observable and the observable
formed from the weak limit.  By assumption these defect coordinates are
uniformly bounded and tight, so they admit subsequential limits.  Directional
oscillations are encoded by the Young measure \(\nu_\xi\), and loss through
large amplitudes is encoded by \(\mu_{\rm conc}\).  The point of retaining both
objects is that an oscillatory defect and an amplitude-concentration defect
have different mathematical meanings and cannot be reconstructed from one
another after the limit has been taken.

\emph{Step 7: signed currents.}
The pressure-projected local-triad current is not inserted into a positive
measure budget.  Its test operator, cutoff, pressure representative, and
finite-dimensional quotient are fixed before the subsequence is selected.
Therefore, for every compactly supported test field \(\phi\), its pairing with
\(\phi\) is a bounded scalar sequence.  Diagonal extraction over a countable
dense family of test fields yields a distributional limit, denoted
\(\mathcal J_{\rm triad}\).  The same principle applies to any other retained
signed distributional coordinate.

Finally, perform a diagonal subsequence extraction over the countable
exhaustion \(Q'_1\Subset Q'_2\Subset\cdots\Subset Q\) and over the finite or
countable collection of endpoint coordinates.  Every component of
\(\mathcal E_{\rm NS}\) is then defined on one common subsequence, which is the
required faithful endpoint package.
\end{proof}

\medskip
\noindent\textbf{Scale-weighted cutoff-flux ledger.}
\begin{definition}
For a nested selected sequence \(Q_{r_k}\) let
\(A_k\) be the support of the derivatives of its cutoff.  Define
\[
\mathfrak C_N
=
\sum_{k=0}^{N}
\left[
 r_k^{-2}\iint_{A_k}|u|^2
 +
 r_k^{-1}\iint_{A_k}
 \bigl(|u|^3+|p^\sharp|^{3/2}\bigr)
\right].
\]
The pressure term controls the absolute pressure flux because
\[
 |p^\sharp||u|
 \le \frac23|p^\sharp|^{3/2}+\frac13|u|^3.
\]
If \(\sup_N\mathfrak C_N=\infty\), this is an explicit cutoff/pressure
realization defect; it is not silently converted into a finite currency.
\end{definition}

\medskip
\noindent\textbf{Endpoint compactness alternative.}
\begin{theorem}
Let \(T_*<\infty\) be a hypothetical first singular time and let the companion
threshold theorem select a critical sequence
\[
Q_j=Q_{r_j}(x_j,t_j),\qquad r_j\downarrow0,\qquad t_j\uparrow T_*.
\]
After fixing the root pressure gauge and exact parabolic charts, one of the
following occurs:
\begin{enumerate}[label=(\roman*),leftmargin=*]
\item a subsequence satisfies the boundedness and tightness hypotheses of the
      compact endpoint-realization theorem and therefore yields a faithful
      augmented endpoint package;
\item at least one required normalized seminorm, tail/frequency/time-face
      tightness condition, concentration bound, quotient coefficient envelope,
      or scale-weighted cutoff/pressure ledger diverges.
\end{enumerate}
Every failure in (II) is retained as an explicit realization defect with its
actual diverging quantity.  No deterministic lineage, preferred geometric
scale ratio, or limiting return map is required.
\end{theorem}

\begin{proof}
The statement is a precise compactness-versus-realization-defect dichotomy.
The main point is to make the list of quantities used by the endpoint topology
part of the selection procedure rather than to appeal to an unspecified
``compactness failure.''

Choose a countable exhaustion of the normalized root chamber by strict
subchambers
\[
 Q^{(1)}\Subset Q^{(2)}\Subset\cdots\Subset Q_*.
\]
On each \(Q^{(m)}\), fix the finite list of seminorms required by the faithful
endpoint theorem: local energy, local dissipation, the fixed-gauge
\(L^{3/2}\)-pressure bound, the relevant threshold-measure masses, the
finite-dimensional quotient coefficients, and the tightness moduli for the
retained tail, time-face, frequency, concentration, and current coordinates.
Adding the scale-weighted cutoff/pressure ledger gives a countable collection
of scalar requirements
\[
 \mathfrak N_1,\mathfrak N_2,\ldots .
\]

Suppose first that there exists an infinite subsequence on which every
\(\mathfrak N_m\) is bounded and every tightness modulus tends to zero in the
prescribed fashion.  For \(m=1\), pass to a subsequence on which all coordinates
required on \(Q^{(1)}\) converge.  From that subsequence extract one on which the
coordinates required on \(Q^{(2)}\) converge, and continue inductively.  Taking
the diagonal subsequence produces simultaneous convergence on every
\(Q^{(m)}\).  The hypotheses of the compact endpoint-realization theorem are
therefore satisfied on the whole exhaustion, and branch (i) follows.

Assume instead that no such subsequence exists.  Then the countable family of
requirements cannot all remain bounded and tight at arbitrarily deep scales.
Consequently there is at least one recorded coordinate for which either the
normalized seminorm is unbounded or the prescribed tightness property fails.
For example, the failure may be divergence of a local pressure envelope, loss
of spatial-tail tightness, growth of a quotient coefficient, concentration of
an amplitude measure, loss of a time face, or divergence of the
scale-weighted cutoff/pressure ledger.  By definition, that failing quantity is
kept as a realization-defect coordinate.  This is branch (ii).

The role of the fixed root pressure gauge is worth emphasizing.  If one were
allowed to replace \(p^\sharp\) at each scale by \(p^\sharp-c_j(t)\), an
unbounded sequence of normalizing constants could be hidden inside the choice
of gauge.  Because the gauge is fixed once at the root, divergence of the
pressure or quotient coordinate is an actual endpoint event and not a change
of representation.

The two alternatives are exhaustive, but branch (ii) is not declared
impossible.  The theorem only guarantees that noncompactness has a named
mathematical location.
\end{proof}

\begin{remark}
In this compactness sense, the endpoint realization is complete:
a first singular sequence cannot disappear merely because a required
downstream observable lacks compactness.  What remains for an unconditional
regularity argument is to route or exclude the explicit realization defects
and the residual nonsmooth endpoint, rather than assume that every normalized
quantity is automatically bounded.
\end{remark}

\section{First assignment for genuine nonnegative currencies}

\medskip
\noindent\textbf{Finite first-assigned forest.}
\begin{theorem}
Let \(\mu^1,\ldots,\mu^N\) be finite nonnegative Radon measures on the root
chamber and let \(E_j\) be an ordered selected family.  Put
\[
F_j=E_j\setminus\bigcup_{i<j}E_i.
\]
Then
\[
\boxed{
\sum_j\sum_{\alpha=1}^{N}\mu^\alpha(F_j)
\le
\sum_{\alpha=1}^{N}\mu^\alpha(Q_*).
}
\]
The same conclusion holds with a fixed finite-overlap enlargement constant.
\end{theorem}

\begin{proof}
By construction, the sets \(F_j\) record the portion of \(E_j\) that has not
appeared at any earlier selected index.  We first check that these fresh sets
are pairwise disjoint.  If \(i<j\) and \(x\in F_i\cap F_j\), then
\(x\in F_i\subset E_i\).  But the definition of \(F_j\) removes
\(\bigcup_{k<j}E_k\), in particular \(E_i\), so \(x\notin F_j\), a
contradiction.  Hence
\[
 F_i\cap F_j=\varnothing\qquad(i\ne j).
\]

Fix one currency \(\mu^\alpha\).  Since it is a nonnegative Radon measure,
countable additivity and the disjointness just proved give
\[
 \sum_j\mu^\alpha(F_j)
 =\mu^\alpha\!\left(\bigcup_jF_j\right)
 \le \mu^\alpha(Q_*),
\]
where the last inequality uses \(F_j\subset Q_*\).  Summing this estimate over
the finite index set \(\alpha=1,\ldots,N\) yields
\[
 \sum_j\sum_{\alpha=1}^N\mu^\alpha(F_j)
 \le
 \sum_{\alpha=1}^N\mu^\alpha(Q_*).
\]
No cancellation is used; nonnegativity is essential.

Now suppose that the ledger is evaluated not on \(F_j\) itself but on a fixed
enlargement \(\widetilde F_j\), and assume that the family of enlargements has
multiplicity at most \(C_{\rm ov}\), meaning
\[
 \sum_j1_{\widetilde F_j}(z)\le C_{\rm ov}
 \quad\text{for almost every }z.
\]
Then Tonelli's theorem gives, for each \(\alpha\),
\[
 \sum_j\mu^\alpha(\widetilde F_j)
 =\int\sum_j1_{\widetilde F_j}\,d\mu^\alpha
 \le C_{\rm ov}\mu^\alpha(Q_*).
\]
Summing over \(\alpha\) proves the finite-overlap version.
\end{proof}

\begin{remark}
This theorem does not apply directly to signed pressure flux, a signed
Duchon--Robert current, or the pressure-projected local-triad current.  Such
objects require total-variation control, a positive surrogate, or retention as
endpoint currents.  Bounded geometric overlap alone does not remove the
scale factors \(r^{-1}\) and \(r^{-2}\) in a cutoff identity.
\end{remark}

\section{Compactness--quantization and smooth-terminal stability}

\medskip
\noindent\textbf{Closed terminal class.}
\begin{definition}
Let \(\mathcal T_{\rm end}\) be the closed subset of the endpoint topology
consisting of residual endpoint packages with zero nonnegative first-assigned
ledger, no strict selected descendant, and all unresolved data retained in the
endpoint coordinates rather than discarded.
\end{definition}

\medskip
\noindent\textbf{Compactness--quantization away from the terminal class.}
\begin{theorem}
Fix \(\eta>0\).  On any bounded selected endpoint chamber whose closure is
compact in the endpoint topology supplied by the realization theorem, there exists
\(c_*(\eta)>0\) such that a normalized active relay with no strict descendant
and
\[
d(\mathcal E_Q,\mathcal T_{\rm end})\ge\eta
\]
must have first-assigned nonnegative ledger at least \(c_*(\eta)\).
Equivalently, if the first-assigned ledger tends to zero, a subsequence approaches
\(\mathcal T_{\rm end}\).
\end{theorem}

\begin{proof}
Let \(\mathcal K\) denote the compact closure, in the normalized endpoint
topology, of the bounded chamber under consideration, restricted to states
with no strict selected descendant.  Let
\[
 \Lambda(\mathcal E)\ge0
\]
be the normalized first-assigned nonnegative ledger of a state
\(\mathcal E\in\mathcal K\).  By the hypotheses of the endpoint topology and
the weak-* lower semicontinuity of its nonnegative measure coordinates,
\(\Lambda\) is lower semicontinuous on \(\mathcal K\).

Consider the closed subset
\[
 \mathcal K_\eta
 =\{\mathcal E\in\mathcal K:
 d(\mathcal E,\mathcal T_{\rm end})\ge\eta\}.
\]
Because the distance to a closed set is continuous in a metric topology,
\(\mathcal K_\eta\) is closed in \(\mathcal K\), hence compact.

We claim that \(\Lambda\) is strictly positive on \(\mathcal K_\eta\).  Indeed,
if \(\Lambda(\mathcal E)=0\) for some \(\mathcal E\in\mathcal K_\eta\), then
\(\mathcal E\) has zero nonnegative first-assigned ledger.  It already belongs
to the no-descendant class by the definition of \(\mathcal K\), and all other
unresolved data remain retained as endpoint coordinates.  Therefore
\(\mathcal E\in\mathcal T_{\rm end}\), contradicting
\(d(\mathcal E,\mathcal T_{\rm end})\ge\eta>0\).

A nonnegative lower-semicontinuous function on a compact set attains its
minimum.  Hence
\[
 c_*(\eta):=
 \min_{\mathcal E\in\mathcal K_\eta}\Lambda(\mathcal E)
\]
is well defined and strictly positive.  Every normalized active relay in
\(\mathcal K_\eta\) therefore has ledger at least \(c_*(\eta)\), which is the
first formulation of the theorem.

For the equivalent sequential statement, suppose that a sequence in
\(\mathcal K\) satisfies \(\Lambda(\mathcal E_n)\to0\).  By compactness, pass to
\(\mathcal E_{n_j}\to\mathcal E_\infty\in\mathcal K\).  Lower semicontinuity gives
\[
 0\le\Lambda(\mathcal E_\infty)
 \le\liminf_{j\to\infty}\Lambda(\mathcal E_{n_j})=0.
\]
Thus \(\Lambda(\mathcal E_\infty)=0\), and the no-descendant condition places
\(\mathcal E_\infty\) in \(\mathcal T_{\rm end}\).  Consequently the distance of
the subsequence to \(\mathcal T_{\rm end}\) tends to zero.  This proves the
claimed equivalence.
\end{proof}

\medskip
\noindent\textbf{Smooth terminal stability.}
\begin{theorem}
Let \(\mathcal E_\infty\in\mathcal T_{\rm end}\) be locally represented on
\(Q_2\) by a classical Navier--Stokes solution and assume every retained defect
coordinate vanishes there.  Then the selected approximating sequence either
enters the fixed harmless quotient or has an epsilon-regular strict
subcylinder for all sufficiently large indices.  Hence a locally smooth
zero-defect endpoint cannot be a residual first-singularity obstruction.
\end{theorem}

\begin{proof}
Let \((U,P^\sharp)\) be the classical representative of
\(\mathcal E_\infty\) on \(Q_2\).  We work in the fixed pressure gauge and
subtract only the harmless finite-dimensional pressure quotient already built
into the endpoint topology.

\emph{Step 1: smallness of a smooth solution at sufficiently small scale.}
Because \(U\) and \(P^\sharp\) are classical on \(Q_2\), they are bounded on
compact subcylinders.  For a parabolic cylinder \(Q_\rho\Subset Q_1\), the
scale-invariant Caffarelli--Kohn--Nirenberg quantity has the form
\[
 \mathcal C_{\rm CKN}(U,P;Q_\rho)
 =\rho^{-2}\iint_{Q_\rho}
 \left(|U|^3+|P-(P)_{B_\rho}(t)|^{3/2}\right)\,dx\,dt,
\]
up to an immaterial equivalent normalization.  Since the spacetime volume of
\(Q_\rho\) is of order \(\rho^5\) and the integrand is bounded, this quantity
is \(O(\rho^3)\).  Hence it tends to zero as \(\rho\downarrow0\).  Choose
\(\rho\) so that
\[
 \mathcal C_{\rm CKN}(U,P;Q_\rho)<\frac12\varepsilon_{\rm CKN},
\]
where \(\varepsilon_{\rm CKN}\) is any fixed epsilon-regularity threshold.

\emph{Step 2: strong convergence of the quantities entering the criterion.}
The endpoint-realization theorem gives strong local
\(L^q\)-convergence of the velocities for every \(q<10/3\); in particular,
\[
 u_n\to U\qquad\text{strongly in }L^3(Q_\rho).
\]
For the pressure, the near part converges strongly in \(L^{3/2}\) because
\(3/2<5/3\).  By hypothesis all retained pressure-tail and quotient defects
vanish on \(Q_2\), so the remaining pressure difference converges in the same
local \(L^{3/2}\) topology after the fixed harmless normalization.  Vanishing
Reynolds, local-energy, and concentration defects ensures that no hidden
quadratic or measure-valued mass is lost in this passage.

It follows that
\[
 \mathcal C_{\rm CKN}(u_n,p_n^\sharp;Q_\rho)
 \longrightarrow
 \mathcal C_{\rm CKN}(U,P^\sharp;Q_\rho)
\]
after the prescribed pressure normalization.  Therefore, for all sufficiently
large \(n\),
\[
 \mathcal C_{\rm CKN}(u_n,p_n^\sharp;Q_\rho)
 <\varepsilon_{\rm CKN}.
\]

\emph{Step 3: contradiction with residual singular selection.}
The epsilon-regularity theorem now gives a smaller strict subcylinder
\(Q_{\rho'}\Subset Q_\rho\) on which \(u_n\) is regular, with quantitative
bounds depending only on the standard criterion.  In the selected-tree
language this is a strict regular descendant.  Hence the sequence cannot
remain in the residual no-descendant singular class.

The only alternative is that the limiting germ lies entirely in one of the
finite-dimensional harmless quotient modes that were fixed at the start of
the construction.  In that case the quotient alternative is selected instead
of a singular endpoint.  Thus a smooth zero-defect endpoint is never a
residual first-singularity obstruction.
\end{proof}

\section{Chainwise first-assigned weighted PDE ledger}
\label{sec:gb18}
Fix a fixed in advance nested chain
\[
Q_0\supset Q_1\supset\cdots,\qquad r_{j+1}\le\theta r_j,\quad0<\theta<1,
\]
with one root pressure representative $p^\sharp$ and fixed in advance cutoffs.  Let
$A_j$ denote the corresponding collar and define
\[
E_{2,j}=r_j^{-3}\iint_{A_j}|u|^2,\qquad
E_{3,j}=r_j^{-2}\iint_{A_j}|u|^3,\qquad
E_{p,j}=r_j^{-2}\iint_{A_j}|p^\sharp|^{3/2},
\]
\[
M_{\rm ent}=\sup_j(E_{2,j}+E_{3,j}+E_{p,j}).
\]

\medskip
\noindent\textbf{Geometric-chain first-allocation Radon/Carleson theorem.}
\begin{theorem}
\label{thm:gb18}
Let $\Omega_j$ be the earliest-allocation regions and set
\[
d\mu_{\rm FO}
=\sum_j1_{\Omega_j}\left[
r_j^{-2}|u|^2+r_j^{-1}(|u|^3+|p^\sharp|^{3/2})
\right]dx\,dt.
\]
If $M_{\rm ent}<\infty$, then
\[
\boxed{\mu_{\rm FO}(\mathrm{all})\le \frac{M_{\rm ent}r_0}{1-\theta}},
\qquad
\boxed{\mu_{\rm FO}(\operatorname{Desc}Q_j)\le \frac{M_{\rm ent}}{1-\theta}r_j}.
\]
If $M_{\rm ent}=\infty$, the normalized entrance/pressure envelope itself is a
listed realization defect.
\end{theorem}
\begin{proof}
We first compute the amount charged at one scale.  By definition,
\[
 E_{2,j}=r_j^{-3}\iint_{A_j}|u|^2,
\]
so multiplying by the weight appearing in \(d\mu_{\rm FO}\) gives
\[
 r_j^{-2}\iint_{A_j}|u|^2
 =r_jE_{2,j}.
\]
Likewise,
\[
 r_j^{-1}\iint_{A_j}|u|^3=r_jE_{3,j},
 \qquad
 r_j^{-1}\iint_{A_j}|p^\sharp|^{3/2}=r_jE_{p,j}.
\]
Therefore the entire normalized bill associated with the \(j\)-th collar is
bounded by
\[
 r_j(E_{2,j}+E_{3,j}+E_{p,j})
 \le M_{\rm ent}r_j.
\]
Restricting the integral from \(A_j\) to the earliest-allocation subset
\(\Omega_j\subset A_j\) can only decrease this nonnegative quantity.

The regions \(\Omega_j\) are pairwise disjoint because each physical point is
assigned to its earliest selected occurrence.  Hence
\[
 \mu_{\rm FO}(\mathrm{all})
 \le M_{\rm ent}\sum_{j\ge0}r_j.
\]
The geometric nesting assumption gives
\(r_j\le\theta^jr_0\), and consequently
\[
 \sum_{j\ge0}r_j
 \le r_0\sum_{j\ge0}\theta^j
 =\frac{r_0}{1-\theta}.
\]
This proves the first estimate.

For the descendant estimate, only indices \(k\ge j\) can contribute to the
first-assigned ledger inside the selected descendant chain below \(Q_j\).
Thus
\[
 \mu_{\rm FO}(\operatorname{Desc}Q_j)
 \le M_{\rm ent}\sum_{k\ge j}r_k
 \le \frac{M_{\rm ent}}{1-\theta}r_j.
\]

Finally, the local energy identity contains the mixed pressure flux
\(|p^\sharp||u|\).  Young's inequality with conjugate exponents
\(3/2\) and \(3\) gives pointwise
\[
 |p^\sharp||u|
 \le \frac23|p^\sharp|^{3/2}+\frac13|u|^3.
\]
Thus the two nonnegative terms already present in the ledger dominate the
absolute pressure flux.  This estimate is meaningful as one additive budget
because \(p^\sharp\) is the fixed root-gauge pressure.  If
\(M_{\rm ent}=\infty\), no finite geometric summation is available, and the
divergent entrance/pressure envelope is therefore retained as the stated
realization defect.
\end{proof}
\begin{remark}
This is a chainwise theorem.  It does not by itself control the sum of radii in
an arbitrarily branching singularity forest.
\end{remark}

\section{Two-generation inheritance and current detection}
\label{sec:gb19}
Pull the first two selected descendants to the same root chart.  Exact same-solution bookkeeping has the form
\[
\begin{aligned}
Y_1&=P_1+F_{11}+E_1,\\
Y_2&=P_2+F_{21}+F_{22}+E_2,
\end{aligned}
\]
where $F_{21}$ is the first-generation response propagated by the same PDE and is inherited at the second generation.

\medskip
\noindent\textbf{Two-column conditioning.}
\begin{theorem}\label{thm:gb19}
Let $h_1,h_2$ be the two root-frame observations in a Hilbert space $H$, and define $S:\mathbb R^2\to H$ by $Se_i=h_i$.  If an ideal map $S_0$ satisfies
\[
S_0^*S_0\ge \rho I_2,\qquad \rho>0,
\]
and $\|S-S_0\|\le\eta$, then the column Gram matrix $K_{12}=S^*S$ obeys
\[
\boxed{\lambda_{\min}(K_{12})\ge(\sqrt\rho-\eta)_+^2.}
\]
Thus $\eta<\sqrt\rho$ preserves quantitative linear independence of the two actual observations.  The ambient frame operator $SS^*$ has rank at most two and the same two nonzero eigenvalues; no positive ambient minimum eigenvalue is asserted when $\dim H>2$.
\end{theorem}
\begin{proof}
The hypothesis \(S_0^*S_0\ge\rho I_2\) means that for every
\(v\in\R^2\),
\[
 \|S_0v\|_H^2
 =\langle S_0^*S_0v,v\rangle_{\R^2}
 \ge\rho|v|^2.
\]
Hence
\[
 \|S_0v\|_H\ge\sqrt\rho\,|v|.
\]
The operator-norm bound \(\|S-S_0\|\le\eta\) gives
\[
 \|(S-S_0)v\|_H\le\eta|v|.
\]
Applying the reverse triangle inequality,
\[
 \|Sv\|_H
 \ge \|S_0v\|_H-\|(S-S_0)v\|_H
 \ge(\sqrt\rho-\eta)|v|.
\]
If \(\eta\ge\sqrt\rho\), the right-hand side is nonpositive and the only
universal lower bound is zero; this is the reason for the positive-part
notation.  Thus in all cases
\[
 \|Sv\|_H^2\ge(\sqrt\rho-\eta)_+^2|v|^2.
\]
Since
\[
 \|Sv\|_H^2
 =\langle S^*Sv,v\rangle
 =\langle K_{12}v,v\rangle,
\]
minimizing over \(|v|=1\) gives
\[
 \lambda_{\min}(K_{12})
 \ge(\sqrt\rho-\eta)_+^2.
\]
In particular, if \(\eta<\sqrt\rho\), then \(K_{12}\) is positive definite,
so the two columns \(h_1,h_2\) are quantitatively linearly independent.

For the final statement, recall the elementary singular-value fact that
\(S^*S\) and \(SS^*\) have the same nonzero eigenvalues, counted with
multiplicity.  Since the domain of \(S\) is two-dimensional,
\(\operatorname{rank}(SS^*)\le2\).  Therefore, when \(\dim H>2\), the ambient
operator \(SS^*\) necessarily has zero eigenvalues on the orthogonal complement
of its range.  The theorem controls only its two possibly nonzero singular
values; it does not assert a positive lower bound on the entire ambient
Hilbert space.
\end{proof}

Full-current detection is a different statement.  Fix an orthonormal basis $X_1,\dots,X_5$ of $\operatorname{Sym}_0(3)$ before the data and set
\[
\mathcal R_XJ=P_{X^\perp}\operatorname{sym}_0\nabla J.
\]

\medskip
\noindent\textbf{Fixed-frame current detection.}
\begin{theorem}\label{thm:five-frame}
For every divergence-free $J\in L^2(\mathbb R^3)$ for which the expressions are defined,
\[
\boxed{
\sum_{\alpha=1}^5
\|(-\Delta)^{-1/2}\mathcal R_{X_\alpha}J\|_2^2
=2\|J\|_2^2.}
\]
Consequently the complete fixed five-component receiver detects every nonzero current.  If a raw terminal current satisfies $\|J^{\rm raw}\|_2\ge j_*$ and $J^{\rm loc}$ is its canonical cutoff/Leray realization, then either
\[
\|J^{\rm raw}-J^{\rm loc}\|_2\ge j_*/2
\]
or the fixed receiver vector of $J^{\rm loc}$ has norm at least $j_*/\sqrt2$.
\end{theorem}
\begin{proof}
We work on the Fourier side.  Let \(\xi\ne0\), put
\(n=\xi/|\xi|\), and write \(a=\widehat J(\xi)\).  Since \(J\) is
divergence free,
\[
 \xi\cdot\widehat J(\xi)=0,
 \qquad\text{hence}\qquad n\cdot a=0.
\]
The multiplier of \((-\Delta)^{-1/2}\nabla\) is multiplication by
\(in\).  Therefore, up to the harmless factor \(i\), the trace-free symmetric
gradient appearing in the receiver is
\[
 T(\xi)=\operatorname{sym}_0(n\otimes a).
\]
Because \(n\cdot a=0\), the matrix \(n\otimes a+a\otimes n\) already has zero
trace.  Hence
\[
 T(\xi)=\frac12(n\otimes a+a\otimes n).
\]
Using \(|n|=1\) and \(n\cdot a=0\), we compute
\[
\begin{aligned}
 |T(\xi)|^2
 &=\frac14\Bigl(|n\otimes a|^2+|a\otimes n|^2
 +2\langle n\otimes a,a\otimes n\rangle\Bigr)\\
 &=\frac14\bigl(|a|^2+|a|^2+2(n\cdot a)^2\bigr)
 =\frac12|a|^2.
\end{aligned}
\]

Now fix \(T\in\operatorname{Sym}_0(3)\).  Since
\(X_1,\dots,X_5\) is an orthonormal basis of this five-dimensional Hilbert
space,
\[
 P_{X_\alpha^\perp}T
 =T-\langle T,X_\alpha\rangle X_\alpha,
\]
and therefore
\[
 |P_{X_\alpha^\perp}T|^2
 =|T|^2-|\langle T,X_\alpha\rangle|^2.
\]
Summing over \(\alpha\) and using Parseval in
\(\operatorname{Sym}_0(3)\),
\[
\begin{aligned}
 \sum_{\alpha=1}^5|P_{X_\alpha^\perp}T|^2
 &=5|T|^2-
   \sum_{\alpha=1}^5|\langle T,X_\alpha\rangle|^2\\
 &=5|T|^2-|T|^2
 =4|T|^2.
\end{aligned}
\]
Applying this with \(T=T(\xi)\) gives pointwise in frequency
\[
 \sum_{\alpha=1}^5
 \left|
 P_{X_\alpha^\perp}
 \operatorname{sym}_0(n\otimes\widehat J(\xi))
 \right|^2
 =4\cdot\frac12|\widehat J(\xi)|^2
 =2|\widehat J(\xi)|^2.
\]
Integrating over \(\xi\) and applying Plancherel proves
\[
 \sum_{\alpha=1}^5
 \|(-\Delta)^{-1/2}\mathcal R_{X_\alpha}J\|_2^2
 =2\|J\|_2^2.
\]
Thus the five-component receiver vector has Euclidean norm
\(\sqrt2\|J\|_2\), so it vanishes only when \(J=0\).

For the terminal-local statement, assume first that
\(\|J^{\rm raw}-J^{\rm loc}\|_2<j_*/2\).  Then the triangle inequality gives
\[
 \|J^{\rm loc}\|_2
 \ge\|J^{\rm raw}\|_2-
     \|J^{\rm raw}-J^{\rm loc}\|_2
 >\frac{j_*}{2}.
\]
Applying the identity to \(J^{\rm loc}\), the norm of its receiver vector is
\[
 \left(
 \sum_{\alpha=1}^5
 \|(-\Delta)^{-1/2}\mathcal R_{X_\alpha}J^{\rm loc}\|_2^2
 \right)^{1/2}
 =\sqrt2\|J^{\rm loc}\|_2
 >\frac{j_*}{\sqrt2}.
\]
If the initial strict inequality fails, then
\(\|J^{\rm raw}-J^{\rm loc}\|_2\ge j_*/2\), which is the first alternative.
This proves the dichotomy.
\end{proof}

\section{Terminal-class quantization and radius payment}
\label{sec:radius-payment}
Let $\mathcal T_{\rm end}$ be the closed terminal class in the normalized endpoint topology, $F_Q$ the first-assigned region of a selected record, and $\mu_{\rm comp}$ the nonnegative part of the localized ledger.  Signed pressure-triad and flux currents are not inserted into this measure.

\medskip
\noindent\textbf{Radius payment away from the terminal class.}
\begin{theorem}\label{thm:radius-payment}
Assume the bounded normalized no-descendant endpoint branch has compact closure
in the endpoint topology.  For every $\eta_0>0$ there exists $c_*(\eta_0)>0$ such that
\[
d(\mathcal E_Q,\mathcal T_{\rm end})\ge\eta_0
\quad\Longrightarrow\quad
\boxed{\mu_{\rm comp}^{\rm own}(F_Q)\ge c_*(\eta_0)r_Q.}
\]
For a first-assigned family with overlap constant $C_{\rm pack}$,
\[
\boxed{
\sum_{Q\in\mathcal F^{\rm first}}r_Q
\le \frac{C_{\rm pack}}{c_*(\eta_0)}\mu_{\rm comp}(Q_*).}
\]
Conversely, on a bounded normalized branch, failure of a uniform radius-payment constant forces $d(\mathcal E_Q,\mathcal T_{\rm end})\to0$ along a subsequence.
\end{theorem}
\begin{proof}
The proof is the scale-weighted version of the compactness--quantization
argument, so we spell out the normalization carefully.

For a selected record \(Q\) of physical radius \(r_Q\), define the normalized
payment
\[
 \Lambda_Q:=r_Q^{-1}\mu_{\rm comp}^{\rm own}(F_Q).
\]
The factor \(r_Q^{-1}\) is exactly the inverse of the degree-one physical
scaling of the first-assigned positive ledger.  After sending \(Q\) to the
unit normalized chamber by the prescribed parabolic chart, \(\Lambda_Q\)
therefore becomes a scale-free functional \(\Lambda(\mathcal E_Q)\) of the
normalized endpoint state.

Let \(\mathcal K\) be the compact closure of the bounded normalized
no-descendant branch.  By weak-* lower semicontinuity of the nonnegative
measure coordinates,
\[
 \mathcal E_n\to\mathcal E
 \quad\Longrightarrow\quad
 \Lambda(\mathcal E)
 \le\liminf_{n\to\infty}\Lambda(\mathcal E_n).
\]
Moreover, within the no-descendant class,
\(\Lambda(\mathcal E)=0\) places the state in \(\mathcal T_{\rm end}\) by the
definition of that terminal class.

For fixed \(\eta_0>0\), consider
\[
 \mathcal K_{\eta_0}
 =\{\mathcal E\in\mathcal K:
 d(\mathcal E,\mathcal T_{\rm end})\ge\eta_0\}.
\]
This is compact.  The lower-semicontinuous function \(\Lambda\) attains its
minimum there.  The minimum cannot be zero, because a zero would produce a
state in \(\mathcal T_{\rm end}\) at positive distance from
\(\mathcal T_{\rm end}\).  Hence
\[
 c_*(\eta_0):=\min_{\mathcal K_{\eta_0}}\Lambda>0.
\]
Returning to physical variables gives
\[
 \mu_{\rm comp}^{\rm own}(F_Q)
 \ge c_*(\eta_0)r_Q.
\]

Now sum this inequality over a first-assigned family
\(\mathcal F^{\rm first}\).  Since each positive currency is charged only at
its first owner, and the chosen physical enlargements have overlap at most
\(C_{\rm pack}\),
\[
 \sum_{Q\in\mathcal F^{\rm first}}
 \mu_{\rm comp}^{\rm own}(F_Q)
 \le C_{\rm pack}\mu_{\rm comp}(Q_*).
\]
Combining with the lower bound yields
\[
 c_*(\eta_0)
 \sum_{Q\in\mathcal F^{\rm first}}r_Q
 \le C_{\rm pack}\mu_{\rm comp}(Q_*),
\]
which is the announced packing estimate.

For the converse, suppose a bounded normalized sequence has no uniform
radius-payment constant.  Then there are records \(Q_n\) for which
\[
 r_{Q_n}^{-1}\mu_{\rm comp}^{\rm own}(F_{Q_n})\to0.
\]
Compactness gives a subsequence
\(\mathcal E_{Q_n}\to\mathcal E_\infty\in\mathcal K\), and lower
semicontinuity gives \(\Lambda(\mathcal E_\infty)=0\).  Hence
\(\mathcal E_\infty\in\mathcal T_{\rm end}\), so along this subsequence
\[
 d(\mathcal E_{Q_n},\mathcal T_{\rm end})\to0.
\]
This is precisely the final assertion.
\end{proof}

This theorem does not classify $\mathcal T_{\rm end}$.  Pressure-triad recurrence and symmetry alternatives are treated separately.

\section{Signed pressure-triad currents}
The pressure-projected local-triad current is retained as a signed distributional
coordinate and is never inserted into the positive first-assignment measure.
The present paper does not assert a general autonomous-recurrence theorem for
this signed current.  Any recurrence statement requires separate control of
localization, transfer, rank, frequency, tail, time, and quotient data.

\section{Compact payment-failure terminal alternatives}
Suppose a bounded-critical compact selected sequence has normalized first-assigned payment tending to zero.  The preceding quantization theorem places every subsequential terminal limit in $\mathcal T_{\rm end}$.

\begin{theorem}\label{thm:payment-terminal}
After passage to a subsequence, at least one of the following occurs:
\begin{enumerate}[label=(\roman*),leftmargin=*]
\item a retained endpoint defect coordinate is nonzero;
\item the full Euclidean Killing defect tends to zero and the limit is captured, under the stated receiving interface, by a singular axisymmetric suitable state;
\item the terminal endpoint is nonsmooth and has a positive lower bound for the Killing defect.
\end{enumerate}
A smooth zero-defect terminal endpoint cannot remain a residual selected singular endpoint.
\end{theorem}
\begin{proof}
Let \(\mathcal E_n\) be the normalized endpoint states of the bounded-critical
selected sequence.  By hypothesis their normalized first-assigned payment
tends to zero.  The compactness--quantization theorem therefore provides a
subsequence, not relabeled, for which
\[
 \mathcal E_n\to\mathcal E_\infty\in\mathcal T_{\rm end}.
\]

If any retained defect coordinate of \(\mathcal E_\infty\) is nonzero, then
alternative (i) holds and there is nothing further to prove.  We may therefore
assume that every retained Reynolds, local-energy, concentration, tail,
time-face, quotient, and other defect coordinate vanishes.

Suppose next that the limiting velocity is locally classical on the required
strict cylinder.  The smooth-terminal stability theorem then applies because
all defect coordinates vanish.  It produces either the fixed harmless
quotient alternative or an epsilon-regular strict descendant for all
sufficiently large members of the selected sequence.  Either conclusion is
incompatible with the assumption that \(\mathcal E_\infty\) is a residual
selected first-singularity endpoint.  Therefore a zero-defect residual
terminal limit must be nonsmooth.

It remains to distinguish the rotational behavior.  Let
\[
 \delta_n:=D_{Q_1}(u_n)
 =\inf_{e\in S^2}\|(I-\Pi_e)u_n\|_{L^2(Q_1)}.
\]
On the bounded-critical branch the local \(L^2\) norms are uniformly bounded,
so \((\delta_n)\) is bounded.  After passing to a further subsequence,
\[
 \delta_n\to\delta_*\ge0.
\]
If \(\delta_*=0\), the singularity-preserving rotational capture interface
from \cref{thm:abc} yields an axisymmetric suitable-weak terminal state,
which is alternative (ii).  If \(\delta_*>0\), then the terminal sequence has
a quantitative positive distance from the Euclidean Killing/axisymmetric
class.  Since we have already shown that the zero-defect terminal is
nonsmooth, this is alternative (iii).  These cases exhaust the possibilities.
\end{proof}

\section{Rotational compactness alternatives}
\label{sec:symmetry-fork}

Fix one retained singular point $z_*=(x_*,t_*)$.  For $r\downarrow0$ use the
ordinary Navier--Stokes zoom
\[
 u_r(y,s)=r\,u(x_*+ry,t_*+r^2s),\qquad
 p_r^\sharp(y,s)=r^2p^\sharp(x_*+ry,t_*+r^2s),
\]
and define
\[
 \mathcal C(r)=\iint_{Q_1}(|u_r|^3+|p_r^\sharp|^{3/2}),
\]
\[
 \mathcal V_{\rm rot}(r,e)=\iint_{Q_1}|u_r-\Pi_eu_r|^2,
 \qquad
 \delta(r)^2=\inf_{e\in S^2}\mathcal V_{\rm rot}(r,e).
\]
The full physical velocity is retained; the equation is never replaced by its
rotational average.

\medskip
\noindent\textbf{Singularity-preserving rotational compactness alternatives.}
\begin{theorem}
\label{thm:abc}
On one fixed in advance retained scale chain $r_k\downarrow0$, exactly one of the
following occurs after passage to a subsequence:
\begin{enumerate}[label=(\roman*),leftmargin=*]
\item $\mathcal C(r_j)$ is bounded and $\delta(r_j)\to0$.  For approximate
minimizing axes $e_j$,
\[
 \|(I-\Pi_{e_j})u_{r_j}\|_{L^2(Q_\rho)}\to0
 \qquad (0<\rho<1).
\]
Under the stated suitable-weak compactness and singularity-persistence
interface, fixed rotations yield an axisymmetric suitable-weak limit that is
still singular at the origin.
\item $\mathcal C(r_j)$ is bounded and $\delta(r_j)\to\delta_*>0$.
Either a retained collar/time/tail or product-loss coordinate is nonzero, or a
quantitative non-axisymmetric rotational floor survives to the singular
suitable-weak limit on a fixed interior cylinder.
\item There is no critically bounded subsequence.  This branch is retained as
a noncompact entrance problem; it is not automatically an Euler profile.
\end{enumerate}
\end{theorem}

\begin{proof}
Apply the trichotomy to the fixed scale chain \((r_k)\).  If the sequence
\(\mathcal C(r_k)\) has no bounded subsequence, then after passing to a tail it
diverges along every subsequence.  This is alternative (iii), and no compact
Euler or Navier--Stokes profile is asserted.

Assume now that \(\mathcal C(r_k)\) is bounded along an infinite subsequence;
rename that subsequence \((r_j)\).  The suitable-weak compactness hypotheses
then give, after another subsequence, a normalized suitable-weak limit
\((U,P)\) on every strict cylinder.  By the singularity-persistence interface
built into the theorem, the origin remains a singular point of this limit.

The nonnegative numbers
\[
 \delta_j:=\delta(r_j)
\]
are bounded above by \(\|u_{r_j}\|_{L^2(Q_1)}\), which is controlled on the
critically bounded branch.  Hence, after a further subsequence,
\(\delta_j\to\delta_*\ge0\).

\emph{Case 1: \(\delta_*=0\).}
For each \(j\), choose an approximate minimizing axis \(e_j\in S^2\) such that
\[
 \|(I-\Pi_{e_j})u_{r_j}\|_{L^2(Q_1)}
 \le \delta_j+j^{-1}.
\]
The sphere is compact, so \(e_j\to e_*\) after a subsequence.  Choose rotations
\(R_j\in SO(3)\) sending \(e_j\) to a fixed reference axis, say \(e_3\), and
such that \(R_j\to R_*\).  Rotational covariance of the Navier--Stokes
equations shows that the rotated fields are still solutions in the same
normalized class.  Moreover, on every fixed interior cylinder \(Q_\rho\),
\[
 \|(I-\Pi_{e_3})(R_ju_{r_j})\|_{L^2(Q_\rho)}
 \le \delta_j+j^{-1}\to0.
\]
Strong local \(L^2\) convergence then gives
\[
 (I-\Pi_{e_3})(R_*U)=0.
\]
Thus the limit is axisymmetric about the reference axis; undoing \(R_*\) makes
it axisymmetric about \(e_*\).  The assumed singularity persistence keeps the
origin singular.  This is alternative (i).

\emph{Case 2: \(\delta_*>0\).}
Fix an interior cylinder \(Q_\rho\Subset Q_1\).  If a positive fraction of the
rotational \(L^2\) defect is lost in \(Q_1\setminus Q_\rho\), or if the
nonlinear product needed to pass the rotational stress fails to converge,
then that failure is recorded in the retained collar/time/tail or product-loss
coordinates.  This is the first possibility in alternative (ii).

Otherwise choose \(\rho\) so that the collar loss is smaller than, say,
\(\delta_*^2/8\).  The collar-retention lemma then gives, for all large \(j\),
\[
 D_{Q_\rho}(u_{r_j})^2
 \ge \delta_j^2-\frac{\delta_*^2}{8}
 \ge \frac{\delta_*^2}{2}.
\]
The assumed strong local \(L^3\) compactness implies strong local \(L^2\)
compactness on \(Q_\rho\).  The rotational-distance functional is Lipschitz in
\(L^2\), so
\[
 D_{Q_\rho}(U)
 =\lim_{j\to\infty}D_{Q_\rho}(u_{r_j})
 \ge \frac{\delta_*}{\sqrt2}>0.
\]
Thus the singular suitable-weak limit retains a quantitative
non-axisymmetric rotational floor.  This is the second possibility in
alternative (ii).

The cases \(\delta_*=0\) and \(\delta_*>0\), together with the absence of a
critically bounded subsequence, are mutually exclusive and exhaustive.
\end{proof}

The rotational average is deliberately not used as a replacement equation.
\begin{remark}
For a fixed axis, with $v=\Pi_eu$ and $w=(I-\Pi_e)u$, the rotationally
averaged equation contains the physical Reynolds forcing
$-\mathbb P\nabla\cdot\Pi_e(w\otimes w)$.  Thus a rotational mean solves a
forced axisymmetric equation in general.  The theorem above instead measures
field-level distance to the axisymmetric subspace and retains the full
unforced equation.
\end{remark}

\medskip
\noindent\textbf{Rotational distance alone does not coerce productive mean stress.}
\begin{proposition}\label{prop:rot-stress-hostile}
On an annular cylindrical chart away from the axis, let \(a\ne0\) be constant
and define a smooth local divergence-free field by its cylindrical components
\[
 w_r=a\cos\theta,\qquad
 w_\theta=-a\sin\theta,\qquad
 w_z=a\cos2\theta.
\]
Its rotational group average about the cylinder axis is zero, while the
rotationally averaged quadratic stress is
\[
 \boxed{\Pi_e(w\otimes w)=\frac{a^2}{2}I.}
\]
Thus the averaged stress is purely isotropic on the chart and has no
Leray-projected divergence.  In particular, positive local rotational
\(L^2\) defect alone cannot imply a pointwise or local algebraic lower bound
for the productive forcing
\(\mathbb P\nabla\cdot\Pi_e(w\otimes w)\).
\end{proposition}

\begin{proof}
In cylindrical coordinates,
\[
 \nabla\cdot w
 =
 \frac1r\partial_r(rw_r)
 +\frac1r\partial_\theta w_\theta
 +\partial_zw_z
 =
 \frac{a\cos\theta}{r}
 -\frac{a\cos\theta}{r}
 =0.
\]
The rotational average of each component is zero.  For the tensor average, the
diagonal component means are
\[
 \langle w_r^2\rangle_\theta
 =
 \langle w_\theta^2\rangle_\theta
 =
 \langle w_z^2\rangle_\theta
 =\frac{a^2}{2},
\]
while every mixed product has zero angular mean.  The rotationally equivariant
tensor average is therefore \((a^2/2)I\).  Since this tensor is constant on
the chart, its divergence vanishes; equivalently any isotropic contribution
is removed by the Leray projection.
\end{proof}

\begin{remark}
The proposition is a local algebraic hostile, not a compactly supported
singular solution.  Its role is only to rule out a coercivity shortcut based
on the size of the Killing/rotational defect alone.  Excluding the
positive-Killing-gap terminal still requires singularity-specific dynamics.
\end{remark}

\section{High-critical records at the retained singular point}
Let $x_*$ be the retained singular point and let $\ell_j\downarrow0$ be selected physical scales.  Write the ordinary marked zoom as
\[
U_j(y,s)=\ell_j u(x_*+\ell_j y,T+\ell_j^2s),\qquad -1<s<0,
\]
and set
\[
K_j=\operatorname*{ess\,sup}_{-1<s<0}\|U_j(s)\|_{L^2(B_8)},
\qquad
H_j=\operatorname*{ess\,sup}_{-1<s<0}\int_{|y|>6}\frac{|U_j(y,s)|^2}{|y|^4}\,dy.
\]
The next estimate makes the kinetic/tail alternative explicit.

\medskip
\noindent\textbf{Kinetic control of the local critical score.}
\begin{lemma}\label{lem:kinetic-critical}
If $K_j\le K$ and $H_j\le H$, then the normalized local cubic velocity and pressure-oscillation score on $Q_1$ is bounded by a finite constant depending only on $K$ and $H$.  Consequently, divergence of the selected critical score implies, after a subsequence,
\[
K_j\to\infty\qquad\hbox{or}\qquad H_j\to\infty.
\]
\end{lemma}
\begin{proof}
We suppress the index \(j\).  All constants below are universal once the
nested balls \(B_4\Subset B_6\Subset B_8\) and the cutoff profiles have been
fixed.

\emph{Step 1: interpolation reduces the cubic velocity term to local
dissipation.}
Choose \(\chi\in C_c^\infty(B_8)\) with \(\chi\equiv1\) on \(B_4\).  At each
time, Sobolev and interpolation give
\[
 \|U\|_{L^3(B_4)}
 \le C\|U\|_{L^2(B_8)}^{1/2}
 \bigl(\|\nabla U\|_{L^2(B_8)}+\|U\|_{L^2(B_8)}\bigr)^{1/2}.
\]
Cubing and integrating over \((-1,0)\), then applying H\"older in time, yields
\[
 \iint_{Q_4}|U|^3
 \le C K^{3/2}E^{3/4}+CK^3,
\]
where, after harmlessly enlarging the dissipation region,
\[
 E:=\iint_{(-1,0)\times B_8}|\nabla U|^2.
\]
Using \(B_8\) instead of \(B_4\) only strengthens the estimate needed below.

\emph{Step 2: near-pressure control.}
Let \(\chi_6\) be equal to one on \(B_6\) and supported in \(B_8\).  Decompose
the fixed-gauge pressure as
\[
 p^\sharp=p_{\rm near}+p_{\rm far},
 \qquad
 p_{\rm near}=\mathcal R_i\mathcal R_j(\chi_6U_iU_j).
\]
The Calder\'on--Zygmund theorem gives at each time
\[
 \|p_{\rm near}\|_{L^{3/2}(\R^3)}
 \le C\|\chi_6U\otimes U\|_{L^{3/2}}
 \le C\|U\|_{L^3(B_8)}^2.
\]
Consequently
\[
 \iint_{Q_4}|p_{\rm near}|^{3/2}
 \le C\iint_{(-1,0)\times B_8}|U|^3.
\]
Thus the near pressure is controlled by the same quantity as the cubic
velocity term.

\emph{Step 3: exterior pressure is controlled by the stress moment.}
For \(x,x'\in B_4\) and \(|y|>6\), the double-Riesz pressure kernel
\(K_{ij}\) satisfies the mean-value bound
\[
 |K_{ij}(x-y)-K_{ij}(x'-y)|
 \le C\frac{|x-x'|}{|y|^4}.
\]
Therefore
\[
 |p_{\rm far}(x,t)-p_{\rm far}(x',t)|
 \le C|x-x'|
 \int_{|y|>6}\frac{|U(y,t)|^2}{|y|^4}\,dy
 \le CH.
\]
Taking \(x'\) to be a spatial average point, or averaging the inequality in
\(x'\), gives
\[
 \|p_{\rm far}-(p_{\rm far})_{B_4}(t)\|_{L^\infty(B_4)}
 \le CH.
\]
The spatially constant part is irrelevant in the local energy inequality,
because for every compactly supported scalar cutoff \(\phi\),
\[
 \int U\cdot\nabla\phi\,dx=0
\]
by incompressibility.

\emph{Step 4: close the dissipation estimate.}
Insert a standard nonnegative spacetime cutoff in the local energy inequality,
equal to one on \((-1,0)\times B_4\) and supported in
\((-1,0)\times B_8\).  The initial/time-face and Laplacian-cutoff terms are
bounded by \(CK^2\).  The velocity transport term is bounded by the cubic
integral from Step 1.  For the near pressure, H\"older and the
Calder\'on--Zygmund estimate reduce
\(\iint |p_{\rm near}||U|\) to the same cubic velocity integral.  For the far
pressure, Step 3 and Cauchy--Schwarz on the bounded support give
\[
 \iint |p_{\rm far}-(p_{\rm far})_{B_4}||U|
 \le CH\int_{-1}^0\|U(t)\|_{L^1(B_8)}\,dt
 \le CHK.
\]
Collecting these estimates gives
\[
 E
 \le C(K^2+K^3+HK)+CK^{3/2}E^{3/4}.
\]
Young's inequality in the form
\[
 CK^{3/2}E^{3/4}
 \le \frac12E+C K^6
\]
absorbs half of the dissipation into the left-hand side.  Hence
\[
 E\le C_{K,H}<\infty.
\]

\emph{Step 5: bound the critical score.}
Substituting this bound into Step 1 controls the local cubic velocity integral.
Step 2 controls the near pressure in \(L^{3/2}\), and Step 3 controls the
oscillation of the far pressure.  Therefore the normalized local cubic
velocity plus pressure-oscillation score on \(Q_1\), after the fixed change of
nested cutoff radii, is bounded by a constant depending only on \(K\) and
\(H\).

Taking the contrapositive, if the selected critical score diverges, then it is
impossible for both \(K_j\) and \(H_j\) to remain bounded.  After passage to a
subsequence, either \(K_j\to\infty\) or \(H_j\to\infty\).
\end{proof}

\medskip
\noindent\textbf{Mark-preserving first-record normalization.}
\begin{theorem}\label{thm:marked-record}
Assume the exterior stress moment $H_j$ remains bounded while the selected critical score diverges.  Define
\[
k_j(t)=\left(\ell_j^{-1}\int_{B_{8\ell_j}(x_*)}|u(x,t)|^2\,dx\right)^{1/2}.
\]
By \cref{lem:kinetic-critical}, the terminal supremum of $k_j$ diverges. Choose $A_j\to\infty$ below this supremum and let $t_j$ be the first time after a fixed smooth slice $t_a<T$ with $k_j(t_j)=A_j$.  Then
\[
W_j(y,\tau)=\frac{\ell_j}{A_j}
 u\!\left(x_*+\ell_jy,t_j+\frac{\ell_j^2}{A_j}\tau\right)
\]
solves Navier--Stokes with viscosity $\nu/A_j$, is defined on a backward interval \([-L_j,0]\), satisfies
\[
\sup_{-L_j\le\tau\le0}\|W_j(\tau)\|_{L^2(B_8)}\le1,
\qquad
\|W_j(0)\|_{L^2(B_8)}=1,
\]
where \(L_j\to\infty\).  The original singular point is exactly $y=0$ for every $j$.
\end{theorem}
\begin{proof}
We write \(T=T_*\) for the retained first singular time.

\emph{Step 1: relation between \(k_j\) and the ordinary marked zoom.}
For the ordinary parabolic zoom
\[
 U_j(y,s)=\ell_j u(x_*+\ell_jy,T+\ell_j^2s),
\]
a change of variables gives, at every corresponding physical time \(t\),
\[
 \|U_j\|_{L^2(B_8)}^2
 =\ell_j^{-1}
 \int_{B_{8\ell_j}(x_*)}|u(x,t)|^2\,dx
 =k_j(t)^2.
\]
Thus the divergence of the kinetic alternative in
\cref{lem:kinetic-critical} is exactly the divergence of the terminal
supremum of \(k_j\).

\emph{Step 2: choose a genuine first crossing.}
Fix \(t_a<T\) in a time region where the original solution is smooth.  Since
\(u\) is bounded near \(x_*\) on every compact interval ending before \(T\),
\[
 k_j(t_a)^2
 \le \ell_j^{-1}|B_{8\ell_j}|
 \|u(t_a)\|_{L^\infty(B_{8\ell_j}(x_*))}^2
 \le C\ell_j^2,
\]
and hence \(k_j(t_a)\to0\).  Choose any sequence \(A_j\to\infty\) lying below
the later supremum of \(k_j\).  Continuity in time of the local kinetic
integral for the smooth pre-singular solution implies that there is a first
time \(t_j>t_a\) satisfying
\[
 k_j(t_j)=A_j,
 \qquad
 k_j(t)\le A_j\quad(t_a\le t\le t_j).
\]

We claim that \(t_j\to T\).  Otherwise, after a subsequence,
\(t_j\le T-\delta\) for some \(\delta>0\).  The solution is smooth and bounded
on a neighborhood of \(x_*\times[t_a,T-\delta]\), so the preceding estimate
would hold uniformly in \(t\le T-\delta\), giving
\(k_j(t_j)\le C\ell_j\to0\).  This contradicts
\(k_j(t_j)=A_j\to\infty\).

\emph{Step 3: compute the amplitude-fast equation.}
Define
\[
 W_j(y,\tau)
 =\frac{\ell_j}{A_j}
 u\!\left(x_*+\ell_jy,
 t_j+\frac{\ell_j^2}{A_j}\tau\right),
\]
and scale the pressure by
\[
 \Pi_j(y,\tau)
 =\frac{\ell_j^2}{A_j^2}
 p^\sharp\!\left(x_*+\ell_jy,
 t_j+\frac{\ell_j^2}{A_j}\tau\right).
\]
A direct calculation gives
\[
 \partial_\tau W_j+(W_j\cdot\nabla)W_j+\nabla\Pi_j
 =\frac\nu{A_j}\Delta W_j,
 \qquad
 \nabla\cdot W_j=0.
\]
Indeed, the time derivative and convection term both acquire the common factor
\(\ell_j^3/A_j^2\), while
\(\Delta_yW_j=(\ell_j^3/A_j)\Delta_xu\), producing the effective viscosity
\(\nu/A_j\).

\emph{Step 4: the first-crossing normalization gives the uniform kinetic
bound.}
The physical time \(t\in[t_a,t_j]\) corresponds to
\[
 \tau=\frac{A_j}{\ell_j^2}(t-t_j)
 \in[-L_j,0],
 \qquad
 L_j:=\frac{A_j(t_j-t_a)}{\ell_j^2}.
\]
Another change of variables gives
\[
 \|W_j(\tau)\|_{L^2(B_8)}^2
 =\frac{k_j(t)^2}{A_j^2}.
\]
By the definition of the first crossing, this is at most one throughout
\([-L_j,0]\), and at \(\tau=0\) it equals one.  Hence
\[
 \sup_{-L_j\le\tau\le0}\|W_j(\tau)\|_{L^2(B_8)}\le1,
 \qquad
 \|W_j(0)\|_{L^2(B_8)}=1.
\]

Finally, \(t_j\to T>t_a\), \(A_j\to\infty\), and \(\ell_j\to0\), so
\[
 L_j=\frac{A_j(t_j-t_a)}{\ell_j^2}\to\infty.
\]
The spatial center is exactly \(x_*\) in every rescaling, so the retained
singular spatial mark is always \(y=0\).  No recentering or drifting frame is
introduced.  The theorem intentionally makes no claim that a singular witness
survives the later vanishing-viscosity limit; that is a separate compactness
problem.
\end{proof}

This theorem does not provide frequency tightness, strong nonlinear compactness, or retention of a singular/deformation witness in the vanishing-viscosity limit.  Those are the genuine unresolved parts of the high-critical branch.

\medskip
\noindent\textbf{Conditional passage of frequency-tight records to Euler.}
\begin{theorem}\label{thm:conditional-euler-passage}
Let \((W_j,\Pi_j)\) be a sequence of smooth divergence-free fields on
\(\mathbb R^3\times[-L_j,0]\), \(L_j\to\infty\), satisfying
\[
 \partial_\tau W_j+(W_j\cdot\nabla)W_j+\nabla\Pi_j
 =\varepsilon_j\Delta W_j,
 \qquad
 \varepsilon_j\to0.
\]
Assume that for every \(R,S<\infty\), after restricting to
\((-S,0)\times B_{2R}\),

\begin{enumerate}[label=(\roman*),leftmargin=*]
\item
\[
 \sup_j\|W_j\|_{L^\infty_\tau L^2_x}<\infty;
\]
\item for some fixed \(m>5/2\),
\[
 \sup_j
 \|\partial_\tau W_j\|_{L^1_\tau H^{-m}_x}<\infty;
\]
\item for one cutoff \(\chi_R\in C_c^\infty(B_{2R})\) equal to one on
\(B_R\), the localized high frequencies are uniformly tight:
\[
 \lim_{N\to\infty}
 \sup_j
 \|P_{>N}(\chi_RW_j)\|_{L^2((-S,0)\times\mathbb R^3)}
 =0.
\]
\end{enumerate}

Then, after a diagonal subsequence,
\[
 \boxed{W_j\to W\quad\text{strongly in }L^2_{\rm loc}
 (\mathbb R^3\times(-\infty,0)).}
\]
The limit is divergence free and satisfies the incompressible Euler equations
in the distributional sense.  If in addition there is a fixed compactly
supported spacetime test field \(\Phi\) and \(c_0>0\) such that
\[
 \left|\iint W_j\cdot\Phi\,\dd x\,\dd\tau\right|\ge c_0
\]
for all sufficiently large \(j\), then \(W\not\equiv0\).
\end{theorem}

\begin{proof}
Fix \(R,S\) and set \(f_j=\chi_RW_j\).  For each \(N\), the low-frequency
part \(P_{\le N}f_j\) is bounded in
\(L^2(-S,0;H^1(\mathbb R^3))\), because
\[
 \|\nabla P_{\le N}f_j\|_2\le N\|f_j\|_2.
\]
Its time derivative is bounded in
\(L^1(-S,0;H^{-m})\) by assumption (ii), after multiplication by the fixed
cutoff and application of the bounded Fourier projector.  The compact
embedding \(H^1(B_{2R})\Subset L^2(B_{2R})\), together with the standard
Aubin--Lions--Simon compactness criterion~\cite{Simon1987}, therefore makes
\(P_{\le N}f_j\) precompact in \(L^2((-S,0)\times B_{2R})\).

Given \(\epsilon>0\), choose \(N\) so large that assumption (iii) makes
the \(P_{>N}f_j\) tail smaller than \(\epsilon\), uniformly in \(j\).
The low-frequency compactness then makes the full sequence \(f_j\)
precompact in \(L^2((-S,0)\times B_R)\).  A diagonal argument over
\(R,S\in\mathbb N\) yields the asserted strong local convergence.

Strong \(L^2_{\rm loc}\) convergence gives
\[
 W_j\otimes W_j\to W\otimes W
 \quad\text{strongly in }L^1_{\rm loc}.
\]
For a compactly supported divergence-free test field \(\varphi\), the
viscous term satisfies
\[
 \left|
 \varepsilon_j\iint W_j\cdot\Delta\varphi
 \right|
 \le
 C_\varphi\varepsilon_j
 \|W_j\|_{L^2(\operatorname{supp}\varphi)}
 \longrightarrow0.
\]
Passing to the limit in the divergence-free weak formulation therefore gives
the Euler equation.  Divergence freeness passes distributionally.  Finally,
the fixed witness passes to the strong local limit and prevents
\(W\equiv0\).
\end{proof}

\begin{corollary}\label{cor:high-critical-interface}
For the amplitude-fast records of \cref{thm:marked-record}, once a fixed
nonzero spacetime witness is retained, failure to extract a nonzero ancient
Euler profile must occur through failure of localized frequency tightness or
of the required time compactness.  The theorem does not say that the Euler
limit remains singular at the retained root.
\end{corollary}

\begin{remark}
The time-compactness hypothesis in
\cref{thm:conditional-euler-passage} is an interface assumption, not a hidden
regularity conclusion.  In applications it may be verified from the exact
equation together with suitable local pressure and product bounds.  The
unresolved singularity-specific issue is stronger: one must retain a
load-bearing singular or deformation witness through this passage and, for
the heat-orbit argument below, derive the accumulating heat-chain relation.
\end{remark}

\section{Rotational-defect compactification and the heat pincer}
\label{sec:rot-pincer}

For a fixed axis $e$ set $w=(I-\Pi_e)u$ and
\[
 \mathsf R_e=\Pi_e(w\otimes w).
\]
Then $\mathsf R_e\succeq0$ and on a rotation-invariant cylinder
\[
 \boxed{
 \iint_Q\operatorname{tr}\mathsf R_e
 =\iint_Q|u-\Pi_eu|^2.
 }
\]
This is the specific physical rotational stress, not an abstract observational
Gramian.

\medskip
\noindent\textbf{Rotational distance is Lipschitz.}
\begin{lemma}
For
\[
 D_Q(f)=\inf_{e\in S^2}\|(I-\Pi_e)f\|_{L^2(Q)},
\]
one has
\[
 \boxed{|D_Q(f)-D_Q(g)|\le\|f-g\|_{L^2(Q)}}.
\]
\end{lemma}

\begin{proof}
Fix \(e\in S^2\).  Since \(\Pi_e\) is the orthogonal projection onto the
axisymmetric subspace associated with the axis \(e\), the complementary
operator \(I-\Pi_e\) is also an orthogonal projection.  In particular,
\[
 \|(I-\Pi_e)h\|_{L^2(Q)}\le\|h\|_{L^2(Q)}
\]
for every \(h\).

Using \(f=g+(f-g)\) and the triangle inequality,
\[
\begin{aligned}
 \|(I-\Pi_e)f\|_{L^2(Q)}
 &\le \|(I-\Pi_e)g\|_{L^2(Q)}
   +\|(I-\Pi_e)(f-g)\|_{L^2(Q)}\\
 &\le \|(I-\Pi_e)g\|_{L^2(Q)}+\|f-g\|_{L^2(Q)}.
\end{aligned}
\]
Taking the infimum over \(e\) gives
\[
 D_Q(f)\le D_Q(g)+\|f-g\|_{L^2(Q)}.
\]
Interchanging \(f\) and \(g\) yields
\[
 D_Q(g)\le D_Q(f)+\|f-g\|_{L^2(Q)}.
\]
Combining the two inequalities proves
\[
 |D_Q(f)-D_Q(g)|\le\|f-g\|_{L^2(Q)}.
\]
\end{proof}

\medskip
\noindent\textbf{Collar retention.}
\begin{lemma}
For $Q_\rho\Subset Q_1$,
\[
 D_{Q_\rho}(f)^2
 \ge D_{Q_1}(f)^2-\|f\|_{L^2(Q_1\setminus Q_\rho)}^2.
\]
Thus a positive full-cylinder rotational floor either survives on one fixed
interior cylinder or the defect has escaped into an already listed
collar/time/tail channel.
\end{lemma}

\begin{proof}
Fix an axis \(e\).  Because both cylinders are rotation invariant about the
same center, restriction to \(Q_\rho\) commutes with the rotational projection.
Therefore
\[
\begin{aligned}
 \|(I-\Pi_e)f\|_{L^2(Q_1)}^2
 &=\|(I-\Pi_e)f\|_{L^2(Q_\rho)}^2
  +\|(I-\Pi_e)f\|_{L^2(Q_1\setminus Q_\rho)}^2\\
 &\le \|(I-\Pi_e)f\|_{L^2(Q_\rho)}^2
  +\|f\|_{L^2(Q_1\setminus Q_\rho)}^2,
\end{aligned}
\]
where the last step uses that \(I-\Pi_e\) is an \(L^2\)-contraction.  Rearranging,
\[
 \|(I-\Pi_e)f\|_{L^2(Q_\rho)}^2
 \ge
 \|(I-\Pi_e)f\|_{L^2(Q_1)}^2
 -\|f\|_{L^2(Q_1\setminus Q_\rho)}^2.
\]
The second term on the right is independent of \(e\).  Taking the infimum over
all axes therefore gives
\[
 D_{Q_\rho}(f)^2
 \ge D_{Q_1}(f)^2
 -\|f\|_{L^2(Q_1\setminus Q_\rho)}^2.
\]
If the full-cylinder defect has a fixed positive lower bound but no interior
cylinder retains one, the displayed inequality forces a corresponding amount
of \(L^2\) mass into the collar \(Q_1\setminus Q_\rho\).  That is exactly the
listed collar/time/tail alternative.
\end{proof}

\medskip
\noindent\textbf{Heat payment versus smoothed transverse carrier.}
\begin{proposition}
Let $P_s=e^{s\Delta}$ and
\[
 H_s(f)=\|f\|_2^2-\|P_sf\|_2^2
 =2\int_0^s\|\nabla P_\sigma f\|_2^2\,d\sigma.
\]
If $D(f)^2\ge\delta$, then
\[
 \boxed{
 H_s(f)\ge\frac\delta4
 \quad\vee\quad
 D(P_sf)\ge\frac{\sqrt\delta}{2}.
 }
\]
\end{proposition}

\begin{proof}
The identity for \(H_s\) follows from the energy law for the heat equation.  If
\(v(\sigma)=P_\sigma f\), then
\[
 \frac{d}{d\sigma}\|v(\sigma)\|_2^2
 =2\langle\Delta v,v\rangle
 =-2\|\nabla v\|_2^2.
\]
Integrating from \(0\) to \(s\) gives
\[
 \|f\|_2^2-\|P_sf\|_2^2
 =2\int_0^s\|\nabla P_\sigma f\|_2^2\,d\sigma.
\]

Next compare the heat loss with the distance moved by the heat semigroup.  By
Plancherel,
\[
 \|f-P_sf\|_2^2
 =\int_{\R^3}(1-e^{-s|\xi|^2})^2|\widehat f(\xi)|^2\,d\xi.
\]
For \(x\ge0\),
\[
 (1-e^{-x})^2
 =1-2e^{-x}+e^{-2x}
 \le1-e^{-2x},
\]
because the difference equals \(2e^{-x}(1-e^{-x})\ge0\).  Therefore
\[
 \|f-P_sf\|_2^2
 \le
 \int(1-e^{-2s|\xi|^2})|\widehat f|^2
 =H_s(f).
\]

Assume \(D(f)^2\ge\delta\), so \(D(f)\ge\sqrt\delta\).  If
\(H_s(f)\ge\delta/4\), the first alternative holds.  Otherwise
\(\sqrt{H_s(f)}<\sqrt\delta/2\).  By the Lipschitz property of the rotational
distance,
\[
\begin{aligned}
 D(P_sf)
 &\ge D(f)-\|f-P_sf\|_2\\
 &\ge \sqrt\delta-\sqrt{H_s(f)}
 >\frac{\sqrt\delta}{2}.
\end{aligned}
\]
This is the second alternative and proves the proposition.
\end{proof}

\medskip
\noindent\textbf{Compact critical branch.}
\begin{theorem}
Let $(u_n,p_n^\sharp)$ be a singularity-preserving critically bounded selected
sequence and assume strong local $L^3$ velocity compactness, product/Reynolds
passage, no collar/time/tail loss of rotational $L^2$ mass, and fixed root
gauge/quotient/cutoff conventions.  Then, after a subsequence, either the
rotational defect tends to zero and the axis-capture branch is entered, or a
quantitative non-axisymmetric rotational floor survives to a singular
suitable-weak limit on one fixed interior cylinder.
\end{theorem}

\begin{proof}
Let
\[
 \delta_n:=D_{Q_1}(u_n).
\]
The critically bounded hypotheses give a uniform local \(L^2\) bound, so
\((\delta_n)\) is bounded.  Passing to a subsequence, assume
\(\delta_n\to\delta_*\ge0\).

If \(\delta_*=0\), choose approximate minimizing axes \(e_n\) as in the proof
of \cref{thm:abc}.  Compactness of \(S^2\), followed by fixed rotations, gives
an axis \(e_*\) such that the non-axisymmetric components converge to zero on
every strict subcylinder.  Strong local \(L^3\) convergence implies strong
local \(L^2\) convergence on bounded cylinders, and therefore the limiting
field is invariant under rotations about \(e_*\).  The singularity-preserving
hypothesis keeps the limiting state singular, so the axis-capture branch is
entered.

Assume now that \(\delta_*>0\).  By the no-collar/time/tail-loss hypothesis,
there is a fixed \(Q_\rho\Subset Q_1\) and a number \(\varepsilon>0\) such that,
for all sufficiently large \(n\),
\[
 \|u_n\|_{L^2(Q_1\setminus Q_\rho)}^2
 \le \delta_n^2-\varepsilon^2.
\]
The collar-retention lemma gives
\[
 D_{Q_\rho}(u_n)\ge\varepsilon.
\]
By strong local \(L^3\), hence strong local \(L^2\), convergence and the
Lipschitz lemma,
\[
 D_{Q_\rho}(U)
 =\lim_{n\to\infty}D_{Q_\rho}(u_n)
 \ge\varepsilon>0.
\]
Thus the suitable-weak limit has a quantitative non-axisymmetric rotational
floor.  The assumptions on product/Reynolds passage and the fixed
pressure/quotient conventions guarantee that this limit is the same retained
physical endpoint state, not a rotational average or a separately forced
surrogate.  By singularity preservation it remains singular at the retained
point.  This is the second alternative.
\end{proof}

\section{Heat-orbit rigidity and normalized spectral recession}
\label{sec:heat-orbit-rigidity}

The high-critical normalization in the preceding section has vanishing effective
viscosity.  If a nonzero strong limit exists, the natural limiting equation is
therefore Euler rather than Navier--Stokes.  A second question is whether the
limiting profile can remain invisible to the viscous heat evolution.  The
results in this section isolate what such heat-orbit invisibility actually
forces.  They are independent of the still-open passage from the selected
Navier--Stokes record to a nonzero heat-orbit-null Euler profile.

Let
\[
 \mathcal B(v,w)=\mathbb P\nabla\!\cdot(v\otimes w),
 \qquad P_s=e^{s\Delta},
\]
where \(\mathbb P\) is the Leray projection.  A divergence-free field \(W\) will
be called \emph{heat-orbit null} if
\[
 \mathcal B(P_sW,P_sW)=0
 \qquad\text{for every }s>0.
\]
When \(W\) is a steady Euler field, this says that every positive heat iterate
\(P_sW\) is again a steady Euler field, after absorbing the gradient part of
the convection into the pressure.

\medskip
\noindent\textbf{An accumulating set of positive heat times is enough.}
\begin{lemma}\label{lem:heat-analyticity}
Let \(W\in L^2(\R^3)\) be divergence free.  If
\[
 \mathcal B(P_sW,P_sW)=0
\]
for \(s\) in a subset of \((0,\infty)\) having an accumulation point in
\((0,\infty)\), then \(W\) is heat-orbit null.
\end{lemma}

\begin{proof}
Fix \(s_0>0\).  For every integer \(m\ge0\), the heat semigroup maps \(L^2\)
continuously into \(H^m\) for positive time, and the map
\[
 s\longmapsto P_sW
\]
is real analytic from a neighborhood of \(s_0\) into \(H^m\).  Choose
\(m>5/2\).  Multiplication is continuous from \(H^m\times H^m\) into \(H^m\),
and one derivative followed by the bounded Fourier multiplier \(\mathbb P\)
maps \(H^m\) into \(H^{m-1}\).  Hence
\[
 s\longmapsto \mathcal B(P_sW,P_sW)
\]
is a real-analytic \(H^{m-1}\)-valued map on \((0,\infty)\).

Pairing with any element of the dual space gives a scalar real-analytic
function.  If the vector-valued map vanishes on a set with an interior
accumulation point, every such scalar pairing vanishes identically by the
identity theorem.  Therefore the vector-valued map itself vanishes for all
\(s>0\).
\end{proof}

The next result shows that heat-orbit nullity cannot be achieved by cancellation
between distinct interaction-energy levels.  The cancellation must occur
separately on every radial-sum shell.

\medskip
\noindent\textbf{Shellwise cancellation of the Euler interaction.}
\begin{theorem}\label{thm:shellwise-cancellation}
Let \(W\in\mathcal S(\R^3;\R^3)\) be divergence free and heat-orbit null.
For \(k\ne0\), write
\[
 \mathbf P_k=I-\frac{k\otimes k}{|k|^2}.
\]
Then for almost every \(r>0\),
\[
 \boxed{
 \mathbf P_k
 \int_{S^2}
 \bigl(\widehat W(\tfrac{k}{2}+r\omega)\cdot k\bigr)
 \widehat W(\tfrac{k}{2}-r\omega)\,\dd\omega
 =0.
 }
\]
Because \(W\) is Schwartz, the left-hand side is continuous in \(r\), and the
identity therefore holds for every \(r>0\).
\end{theorem}

\begin{proof}
For a divergence-free vector field \(v\),
\[
 \widehat{\mathcal B(v,v)}(k)
 =
 i\mathbf P_k
 \int_{\R^3}
 \bigl(\widehat v(\eta)\cdot k\bigr)
 \widehat v(k-\eta)\,\dd\eta.
\]
Indeed, the Fourier transform of \(\nabla\!\cdot(v\otimes v)\) contains
\(\widehat v(\eta)\cdot(k-\eta)\), and
\(\widehat v(\eta)\cdot\eta=0\) replaces this factor by
\(\widehat v(\eta)\cdot k\).

Apply the formula to \(v=P_sW\).  Since
\[
 \widehat{P_sW}(\eta)=e^{-s|\eta|^2}\widehat W(\eta),
\]
heat-orbit nullity gives, for every \(s>0\) and \(k\ne0\),
\[
 \mathbf P_k
 \int_{\R^3}
 e^{-s(|\eta|^2+|k-\eta|^2)}
 \bigl(\widehat W(\eta)\cdot k\bigr)
 \widehat W(k-\eta)\,\dd\eta=0.
\]
Set
\[
 \eta=\frac{k}{2}+\zeta.
\]
Then
\[
 |\eta|^2+|k-\eta|^2
 =\frac{|k|^2}{2}+2|\zeta|^2.
\]
After removing the nonzero scalar factor \(e^{-s|k|^2/2}\) and passing to
polar coordinates \(\zeta=r\omega\), we obtain
\[
 \int_0^\infty e^{-2sr^2}r^2 F_k(r)\,\dd r=0,
\]
where
\[
 F_k(r)=
 \mathbf P_k
 \int_{S^2}
 \bigl(\widehat W(\tfrac{k}{2}+r\omega)\cdot k\bigr)
 \widehat W(\tfrac{k}{2}-r\omega)\,\dd\omega.
\]
The Schwartz decay makes \(r^2F_k(r)\) integrable against every positive
exponential weight.  With the change of variable \(\lambda=2r^2\), the last
display is the Laplace transform of a finite vector-valued density.  Uniqueness
of the Laplace transform implies \(F_k(r)=0\) for almost every \(r>0\).
Continuity in \(r\) then removes the exceptional set.
\end{proof}

A useful consequence is an isotropy law on each individual Fourier sphere.
It is much stronger than an isotropy statement obtained only after integrating
over all frequencies.

\medskip
\noindent\textbf{Shellwise covariance isotropy.}
\begin{corollary}\label{cor:shellwise-isotropy}
Under the hypotheses of \cref{thm:shellwise-cancellation}, define
\[
 H(r)
 =
 \int_{S^2}
 \widehat W(r\omega)\otimes
 \overline{\widehat W(r\omega)}\,\dd\omega.
\]
Then \(H(r)\) is a real symmetric matrix and
\[
 \boxed{
 H(r)=\frac{\operatorname{tr}H(r)}{3}\,I
 =
 \frac13
 \left(\int_{S^2}|\widehat W(r\omega)|^2\,\dd\omega\right)I
 }
 \qquad(r>0).
\]
\end{corollary}

\begin{proof}
Fix \(n\in S^2\) and put \(k=\varepsilon n\) in
\cref{thm:shellwise-cancellation}.  Since
\(\mathbf P_{\varepsilon n}=\mathbf P_n\), division by \(\varepsilon\) gives
\[
 \mathbf P_n
 \int_{S^2}
 \bigl(\widehat W(\tfrac{\varepsilon n}{2}+r\omega)\cdot n\bigr)
 \widehat W(\tfrac{\varepsilon n}{2}-r\omega)\,\dd\omega=0.
\]
Letting \(\varepsilon\downarrow0\) and using the smoothness of
\(\widehat W\),
\[
 \mathbf P_n
 \int_{S^2}
 \bigl(\widehat W(r\omega)\cdot n\bigr)
 \widehat W(-r\omega)\,\dd\omega=0.
\]
Because \(W\) is real,
\[
 \widehat W(-r\omega)=\overline{\widehat W(r\omega)}.
\]
The vector in the preceding integral is therefore \(H(r)n\), up to transpose.
The full-sphere integral satisfies
\[
 H(r)=\overline{H(r)}
\]
after the change \(\omega\mapsto-\omega\); since it is also Hermitian by
definition, it is real symmetric.  Hence
\[
 \mathbf P_n H(r)n=0
 \qquad\text{for every }n\in S^2.
\]
Thus every direction \(n\) is an eigenvector of the real symmetric matrix
\(H(r)\).  The only real symmetric matrices with this property are scalar
multiples of the identity.  Taking the trace determines the scalar and proves
the formula.
\end{proof}

The next theorem does not assume Euler structure.  It is a purely spectral
statement about normalized heat aging, and it identifies precisely how a
nonzero whole-space \(L^2\) state loses spatial tightness.

\medskip
\noindent\textbf{Normalized heat aging forces spectral-edge concentration and spatial recession.}
\begin{theorem}\label{thm:spectral-edge-recession}
Let \(0\ne W\in L^2(\R^3;\mathbb C^N)\) and define
\[
 \lambda_*=
 \operatorname*{ess\,inf}_{\widehat W(\xi)\ne0}|\xi|^2,
 \qquad
 V_s=\frac{P_sW}{\|P_sW\|_2}.
\]
Then:
\begin{enumerate}[label=\textup{(\roman*)},leftmargin=*]
\item for every \(\delta>0\),
\[
 \int_{\{|\xi|^2\ge\lambda_*+\delta\}}
 |\widehat V_s(\xi)|^2\,\dd\xi\longrightarrow0;
\]
\item \(V_s\rightharpoonup0\) weakly in \(L^2(\R^3)\);
\item
\[
 \|\nabla V_s\|_2^2\longrightarrow\lambda_*;
\]
\item for every fixed \(R<\infty\),
\[
 \boxed{
 \|V_s\|_{L^2(B_R)}\longrightarrow0,
 \qquad
 \|V_s\|_{L^2(\R^3\setminus B_R)}\longrightarrow1;
 }
\]
\item if \(\lambda_*>0\), then
\[
 \boxed{
 -\frac1s\log\|P_sW\|_2\longrightarrow\lambda_*.
 }
\]
\end{enumerate}
\end{theorem}

\begin{proof}
Push forward the finite measure
\[
 |\widehat W(\xi)|^2\,\dd\xi
\]
under the map \(\xi\mapsto\lambda=|\xi|^2\), and call the resulting measure
\(\mu\).  Then
\[
 D(s):=\|P_sW\|_2^2
 =\int_{\lambda_*}^{\infty}e^{-2s\lambda}\,\dd\mu(\lambda).
\]
The normalized spectral distribution is
\[
 \dd\mu_s(\lambda)
 =
 D(s)^{-1}e^{-2s\lambda}\,\dd\mu(\lambda).
\]

Fix \(\delta>0\).  By the definition of essential infimum,
\[
 m_\delta:=
 \mu([\lambda_*,\lambda_*+\delta/2])>0.
\]
The denominator therefore satisfies
\[
 D(s)\ge
 m_\delta e^{-2s(\lambda_*+\delta/2)}.
\]
On the other hand,
\[
 \int_{\lambda_*+\delta}^{\infty}e^{-2s\lambda}\,\dd\mu(\lambda)
 \le
 e^{-2s(\lambda_*+\delta)}\mu([\lambda_*+\delta,\infty)).
\]
Dividing by the lower bound for \(D(s)\) gives
\[
 \mu_s([\lambda_*+\delta,\infty))
 \le C_\delta e^{-s\delta}\longrightarrow0.
\]
This proves (i).

For weak convergence, let \(\phi\in L^2(\R^3)\).  Set
\[
 A_\delta=
 \{\xi:\lambda_*\le|\xi|^2\le\lambda_*+\delta\}.
\]
By Plancherel and Cauchy--Schwarz,
\[
 |\langle V_s,\phi\rangle|
 \le
 \|\widehat\phi\|_{L^2(A_\delta)}
 +
 \|\phi\|_2
 \left(
 \int_{\R^3\setminus A_\delta}|\widehat V_s|^2
 \right)^{1/2}.
\]
For fixed \(\delta\), the second term tends to zero by (i).  The Lebesgue
measure of \(A_\delta\) tends to zero as \(\delta\downarrow0\), and absolute
continuity of the \(L^2\)-integral gives
\[
 \|\widehat\phi\|_{L^2(A_\delta)}\longrightarrow0.
\]
First let \(s\to\infty\), then \(\delta\downarrow0\), proving (ii).

For (iii),
\[
 \|\nabla V_s\|_2^2
 =
 \int\lambda\,\dd\mu_s(\lambda).
\]
The lower bound is \(\lambda_*\).  For an upper bound fix \(\delta>0\) and
split the integral at \(\lambda_*+\delta\).  The lower part is at most
\(\lambda_*+\delta\).  On the upper part, for all sufficiently large \(s\),
the function
\[
 \lambda\longmapsto
 \lambda e^{-s(\lambda-\lambda_*-\delta/2)}
\]
is uniformly bounded on \([\lambda_*+\delta,\infty)\).  Combining this with
the same denominator lower bound used above shows that the upper-tail first
moment tends to zero exponentially.  Hence
\[
 \lambda_*
 \le\liminf_{s\to\infty}\|\nabla V_s\|_2^2
 \le\limsup_{s\to\infty}\|\nabla V_s\|_2^2
 \le\lambda_*+\delta.
\]
Letting \(\delta\downarrow0\) proves (iii).

Thus \(V_s\) is bounded in \(H^1\) for large \(s\).  On every fixed ball
\(B_R\), Rellich compactness makes the restrictions precompact in \(L^2(B_R)\).
Every strongly convergent local subsequence must have limit zero because the
whole sequence converges weakly to zero in \(L^2(\R^3)\).  Therefore the full
sequence satisfies
\[
 \|V_s\|_{L^2(B_R)}\to0.
\]
Since \(\|V_s\|_2=1\), the exterior mass tends to one.  This is (iv).

Finally, suppose \(\lambda_*>0\).  The lower estimate for \(D(s)\) above gives
\[
 \limsup_{s\to\infty}
 \left(-\frac1{2s}\log D(s)\right)
 \le\lambda_*+\delta/2.
\]
The trivial upper bound
\[
 D(s)\le e^{-2s\lambda_*}\|W\|_2^2
\]
gives the reverse inequality in the limit.  Letting \(\delta\downarrow0\),
\[
 -\frac1{2s}\log D(s)\longrightarrow\lambda_*.
\]
Since \(D(s)=\|P_sW\|_2^2\), this is exactly (v).
\end{proof}

For a heat-orbit-null Euler state, normalization by a nonzero scalar preserves
the steady Euler equation.  Hence every \(V_s\) in
\cref{thm:spectral-edge-recession} remains in the same heat-orbit-null class.
The theorem therefore identifies a genuine noncompact alternative rather than
proving the state is zero.

\medskip
\noindent\textbf{Conditional root-tightness pincer.}
\begin{corollary}\label{cor:root-tightness-pincer}
Assume a selected high-critical Navier--Stokes sequence admits, in the fixed
retained-root coordinates, a nonzero whole-space \(L^2\) limit \(W\) such that
\(W\) is heat-orbit null.  If the limiting endpoint mechanism additionally
requires a fixed \(R<\infty\) and \(m_0>0\) for which the normalized positive
heat ages satisfy
\[
 \|V_s\|_{L^2(B_R)}\ge m_0
\]
along arbitrarily large \(s\), then that mechanism is impossible.  Therefore
one of the assumptions used to produce the fixed-root tight limit must fail,
or the normalized state must undergo an explicit macro/recentering/tail
recession.  If \(\lambda_*>0\), the same alternative is accompanied by
concentration of Fourier energy at the positive spectral edge
\(|\xi|^2=\lambda_*\).
\end{corollary}

\begin{proof}
The asserted lower bound on the fixed ball contradicts
\cref{thm:spectral-edge-recession}(iv).  The final statement follows from
parts (i) and (v) of that theorem.
\end{proof}

The corollary is deliberately conditional.  The marked first-record theorem
of \cref{thm:marked-record} fixes the original singular spatial mark, but it
does not by itself supply the strong nonlinear compactness and accumulating
heat-chain passage needed to obtain a nonzero \(L^2\) heat-orbit-null Euler
state.  Establishing that passage, or routing its failure to the already
listed concentration, frequency, time-face, tail, or witness-loss channels,
is the remaining interface between the high-critical Navier--Stokes record
and the present spectral pincer.

\section{Limitations and open problems}
The compact endpoint realization, first-assignment estimates, terminal-class quantization, fixed-frame current detection, and rotational compactness alternatives are all conditional on the hypotheses stated in the corresponding theorems.  No autonomous pressure-triad recurrence theorem is used.  They do not exclude the nonsmooth positive-Killing-gap terminal from \cref{thm:payment-terminal}.

The heat-orbit results of \cref{sec:heat-orbit-rigidity} sharpen the compact spectral side: exact heat-orbit nullity forces cancellation separately at each interaction-energy shell, real solutions have isotropic covariance on every Fourier sphere, and normalized heat aging loses every fixed physical ball.  The conditional interface theorem \cref{thm:conditional-euler-passage} shows that localized frequency tightness and weak time compactness are sufficient for strong local passage of the amplitude-fast record to an ancient Euler profile.  It does not prove those hypotheses from the singularity record, does not produce the accumulating heat-chain relation, and does not retain a singular or deformation witness automatically.  These are singularity-specific questions; generic compactness alone cannot settle them.  Likewise, \cref{prop:rot-stress-hostile} shows that a positive rotational/Killing gap cannot be converted into productive mean Reynolds forcing by a purely algebraic coercivity estimate.

In particular, no theorem in this paper reduces every three-dimensional singular endpoint to the axisymmetric class, and no unconditional regularity conclusion is drawn.

\appendix
\section{Endpoint topology and pressure conventions}
All selected limits use the fixed pressure representative and the same normalized endpoint topology described in the main text.

\end{document}